\documentclass{birkmult}
\usepackage{latexsym}
\usepackage{amssymb}
\usepackage{amsmath}
\usepackage{mathrsfs}
\usepackage{graphicx}
\usepackage{verbatim}
\usepackage[all]{xy}

\newtheorem{theo}{Theorem}[section]
\newtheorem{example}[theo]{Example}
\newtheorem{lemma}[theo]{Lemma}

\newtheorem{corollary}[theo]{Corollary}
\newtheorem{prop}[theo]{Proposition}

\newcommand {\ZZ} {\mathbb {Z}}

\renewcommand {\AA} {\mathbf{A}}
\newcommand {\CC} {\mathbb {C}}

\newcommand {\End} {\mathrm{End}}
\newcommand {\Top} {\mathrm{top}}
\newcommand {\Ker} {\mathrm{Ker}}
\newcommand {\Id} {\mathrm{Id}}
\newcommand {\card} {\mathrm {card}}

\renewcommand{\tt}{\mathfrak{t}}

\newcommand{\hh}{\mathfrak{h}}

\newcommand{\cc}{\mathfrak{c}}
\renewcommand{\gg}{\mathfrak{g}}
\newcommand{\bb}{\mathfrak{b}}

\renewcommand{\ll}{\mathfrak{l}}

\renewcommand {\phi} {\varphi}

\newcommand{\refth}[1]{Theorem \ref{#1}}
\newcommand{\refex}[1]{Example \ref{#1}}
\newcommand{\refle}[1]{Lemma \ref{#1}}

\newcommand{\refcor}[1]{Corollary \ref{#1}}
\newcommand{\refprop}[1]{Proposition \ref{#1}}

\newcommand{\refeq}[1]{(\ref{#1})}

\renewcommand{\Im}{\mathrm{Im}}
\newcommand{\im}{\mathrm{Im}}
\renewcommand{\span}{\mathrm{span}}

\renewcommand{\sl}{\mathrm{sl}}

\newcommand{\ext}{\mathrm{Ext}}
\newcommand{\fin}{\mathrm{fin}}

\newcommand{\ds}{\displaystyle}

\def\cplus{\hbox{$\subset${\raise0.3ex\hbox{\kern -0.55em ${\scriptscriptstyle +}$}}\ }}

\def\clplus{\hbox{$\subset${\raise0.3ex\hbox{\kern -0.55em ${\scriptscriptstyle +}$}}\ }}
\def\crplus{\hbox{$\supset${\raise1.05pt\hbox{\kern -0.55em ${\scriptscriptstyle +}$}}\ }}

\newcommand{\Int}{\mathrm{Int}}
\newcommand{\Intwt}{\Int^{\mathrm{wt}}_{\gg,\hh}}
\newcommand{\Intfin}{\Int^{\mathrm{fin}}_{\gg,\hh}}
\newcommand{\tens}{\mathrm{Tens}}
\newcommand{\lind}{\widetilde{\mathrm{Tens}}}

\newcommand{\soc}{\mathrm{soc}}
\renewcommand{\hom}{\mathrm{Hom}}
\renewcommand{\abstractname}{\noindent Summary}

\begin{document}

\title{Categories of integrable $sl(\infty)$-, $o(\infty)$-, $sp(\infty)$-modules \\}

\author{Ivan Penkov}

\address{Jacobs University Bremen, School of Engineering and Science,
Campus Ring 1, 28759 Bremen, Germany}

\email{i.penkov@jacobs-university.de}

\author{Vera Serganova}

\address{Department of Mathematics, University of California
Berkeley, Berkeley CA 94720, USA}

\email{serganov@math.berkeley.edu}

\subjclass{Primary 17B65, 17B10}

\keywords{Locally finite Lie algebra, tensor representation, integrable module, Loewy length, semisimplicity, injective hull, injective resolution.}

\date{June 3, 2010}

\begin{abstract}
We investigate several categories of integrable $\sl(\infty)$-, $o(\infty)$-, $sp(\infty)$-modules. In particular, we prove that the category of integrable $\sl(\infty)$-, $o(\infty)$-, $sp(\infty)$-modules with finite-dimensional weight spaces is semisimple. The most interesting category we study is the category $\lind_\gg$ of tensor modules. Its objects $M$ are defined as integrable modules of finite Loewy length such that the algebraic dual $M^*$ is also integrable and of finite Loewy length.

We prove that the simple objects of $\lind_\gg$ are precisely the simple tensor modules, i.e. the simple subquotients of the tensor algebra of the direct sum of the natural and conatural representations.

We also study injectives in $\lind_\gg$ and compute the Ext$^1$'s between simple modules. Finally, we characterize a certain subcategory $\tens_\gg$ of $\lind_\gg$ as the unique minimal abelian full subcategory of the category of integrable modules which contains a non-trivial module and is closed under tensor product and algebraic dualization.
\end{abstract}

\maketitle
\section{Introduction}

The category of finite-dimensional representations of a Lie
algebra is endowed with a natural contravariant involution
\begin{equation}\label{eq1}
M\rightsquigarrow M^{*},
\end{equation}
where $^*$ indicates dual space. For categories of
infinite-dimensional modules $(1)$ is never an
involution as $M\not\simeq M^{**}$. This is why one usually looks
for a ``restricted dual'' or a ``continuous dual'' which might
still yield a contravariant involution on a given category of
infinite-dimensional modules. In this paper we study two
categories of infinite-dimensional modules of certain
infinite-dimensional Lie algebras and show, in particular, that
there exists an interesting category $\lind_\gg$ of
infinite-dimensional representations on which the functor (1) of
algebraic dualization is well-defined and preserves the property
of a module to be of finite Loewy length. 

More precisely, we study representations of locally finite Lie
algebras, i.e. of direct limits of finite-dimensional Lie
algebras. There are three well-known classical simple locally
finite Lie algebras $sl(\infty), o(\infty), sp(\infty)$, each of
them being defined by an obvious direct limit. None of these Lie
algebras admits non-trivial finite-dimensional representations,
and instead one studies integrable representations (the definition
see in section 2 below). However, the category of integrable
$\gg$-modules is vast (and ``wild'' in the technical sense), so it
is reasonable to look for interesting subcategories.

One subcategory we study is the category of integrable weight
modules with finite-dimensional weight spaces, and this is obviously an
analog of the category of finite-dimensional representations of
a classical finite-dimensional Lie algebra. It is less obvious
that for $\ds \gg=sl(\infty)$ this category contains some rather
interesting simple modules which are not highest weight modules.
The first main result of this paper is the proof of the
semisimplicity of this category: an extension of Hermann Weyl's
semisimplicity theorem to the classical Lie algebras $sl(\infty),
o(\infty), sp(\infty)$.

The above category is clearly not the only reasonable
generalization of the category of finite-dimensional
representations, as for instance it does not contain the adjoint
representation. Indeed, note that the adjoint representation has
an infinite-dimensional weight space, the Cartan subalgebra
itself. On the other hand, the adjoint representation is naturally
a simple tensor module as defined in [PS]. More generally, we
define the category $\lind_\gg$ for $\gg\cong
sl(\infty),\,o(\infty),\,sp(\infty)$ simply as the largest
category of integrable $\gg$-modules which is closed under
algebraic dualization and such that every object has finite Loewy
length. This category is a (non-rigid) tensor category with respect
to the usual tensor product. 

The second main contribution of the
present paper is the study of the category $\lind_\gg$. In particular,
we study injectives in $\lind_\gg$ and compute the Ext$^1$'s between
simple modules. We also
give an alternative characterization of $\lind_\gg$ by proving
that an integrable $\gg$-module is an object of $\tens_\gg$
if and only if it has finite Loewy length and admits only finitely
many non-isomorphic simple subquotients each of which is a submodule
of a suitable finite tensor product of natural and conatural modules.
Finally, we describe a certain subcategory $\tens_\gg$ of $\lind_\gg$
as the unique minimal abelian full subcategory of the category of
integrable modules which contains a non-trivial module and is closed
under tensor product and algebraic dualization.

\textbf{Acknowledgement.} We thank Gregg Zuckerman for his
supportive interest and constructive criticism of this project.
We thank also A. Baranov for pointing out the connection of
Proposition 4.3 to his work.
Both authors acknowledge partial support through DFG Grants PE
980/2-1 and PE 980/3-1, and the second author aknowledges partial
support through NSF grant 0901554.

\section{Basic definitions}

\hspace{0.5 cm}The ground field is $\mathbb{C}$ and $\otimes$
stands for $\otimes_\CC$. If $\mathcal C$ is a category, $C\in
\mathcal C$ indicates that $C$ is an object of $\mathcal C$. If
$P$ is a set, we denote by $2^P$ the power set of $P$. We recall
that the cardinal numbers $\beth_n$ are defined inductively:
$\beth_0=\card\,\ZZ$, $\beth_1=\card\,2^\ZZ$,
$\beth_n=\card\,2^{P_{n-1}}$, where $P_{n-1}$ is a set of
cardinality $\beth_{n-1}$.

In this paper $\gg$ stands for a \textit{locally semisimple}
(complex) Lie algebra. By definition,
$\gg=\bigcup_{i\in\mathbb{Z}_{>0}}\gg_i$ where
\begin{equation}\label{eq2}
\gg_1\subset\gg_2\subset\gg_3\subset\dots
\end{equation}
 is a sequence of inclusions of semisimple finite-dimensional Lie
 algebras. We call the sequence (2) an {\it exhaustion} of $\gg$, and we will assume that it is fixed.
 A locally semisimple Lie algebra is \textit{locally simple} if it
 admits an exhaustion (2) so that all $\gg_i$ are simple. It is clear that a locally simple Lie algebra is
 simple. If no restrictions on $\gg$ are clearly stated, in what
 follows $\gg$ is assumed to be an arbitrary locally semisimple
 Lie algebra.

 A locally simple algebra $\gg$ is \textit{diagonal} if an exhaustion (2) can be chosen so that all $\gg_i$ are classical simple Lie algebras and the natural representation $V_i$ of $\gg_i$, when restricted to $\gg_{i-1}$, has the form $k_iV_{i-1}\oplus l_iV_{i-1}^*\oplus\mathbb{C}^{s_i}$ for some $k_i, l_i$ and $s_i\in\mathbb{Z}_{\geq 0}$. Here $V_{i-1}$ stands for the natural representation of $\gg_{i-1}$, $\mathbb{C}^{s_i}$ stands for the trivial module of dimension $s_i$, and $\ds k_iV_{i-1}$ (respectively, $l_iV_{i-1}^*$) denotes the direct sum of $k_i$(respectively, $l_i$) copies of $V_{i-1}$ (respectively, $V_{i-1}^*$).

The three classical simple Lie algebras $sl(\infty), o(\infty)$
and $sp(\infty)$ (defined respectively as $\displaystyle
 sl(\infty)=\cup_i sl(i)$,  $\displaystyle o(\infty)=\cup_i
 o(i)$, $\ds sp(\infty):=\cup_isp(2i)$ via the natural inclusions
 $sl(i)\subset sl(i+1))$ etc.) are clearly diagonal. Moreover,
 $sl(\infty), o(\infty),sp(\infty)$ are (up to isomorphism) the only finitary locally simple Lie algebras $\gg$; \textit{finitary}
means
 by definition that $\gg$ admits a faithful countable-dimensional
 $\gg$-module with a basis in which each element $g\in\gg$ acts
 through a finite matrix,  [Ba1], [Ba3].  More generally, there exists
 also a classification of locally simple diagonal Lie algebras up to
 isomorphism, [BZh]. We do not use this classification in the present
 paper and present only the simplest example of a diagonal Lie algebra
 not isomorphic to $sl(\infty), o(\infty)$ or $sp(\infty)$. This is
 the Lie algebra $sl(2^{\infty})$ defined as the direct limit
 $\ds\lim_{\to}$ $  sl(2^i)$ under the inclusions

\[ sl(2^i) \rightarrow sl(2^{i+1}), A\rightarrow\left( \begin{array}{cc}
A & 0  \\
0 & A  \end{array} \right).\]

A $\gg$-module $M$ is \textit{integrable} if 
$\dim\span\{m, g\cdot m,g\cdot m^2,\cdots\}<\infty$ for any $m\in M$ and $g\in\gg$. 
Since $\gg$ is locally semisimple, this is equivalent to the condition that,
 when restricted to any semisimple finite-dimensional subalgebra $\mathfrak{f}$ of $\gg$, $M$ is isomorphic to a (not necessarily countable) direct sum of finite-dimensional $\mathfrak{f}$-modules. We denote by $\Int_\gg$ the category of integrable $\gg$-modules; $\Int_\gg$
is a full subcategory of the category of $\gg$-modules $\gg$-mod.

Any countable-dimensional $\gg$-module $M\in\Int_{\gg}$ can be
exhausted by finite dimensional $\gg_i$-modules $M_i$, i. e. there
exists a chain of finite-dimensional $\gg_i$-submodules
$M_1\subset M_2\subset\dots$ such that $\ds
M=\underrightarrow{\lim}M_i$. We call $M$ {\it locally simple} if
all $M_i$ can be chosen to be simple modules. It is clear that a
locally simple module is simple. Note also that if $M$ is locally
simple then any two exhaustions $\{M_i\}$ and $\{M'_i\}$ coincide
from some point on: that follows from the fact that $M_i\cap
M'_i\neq 0$ for some $i$ and hence $M_j=M'_j=M_j\cap M'_j$ for any
$j\geq i$. We say that a locally simple $\gg$-module $\ds
M=\underrightarrow{\lim}M_i$ is a \emph{highest weight module} if
there is a chain of nested Borel subalgebras $\bb_i$ of $\gg_i$
such that the $\bb_i$-highest weight space of $M_i$ is mapped into
the $\bb_{i+1}$-highest weight space of $M_{i+1}$ under the
inclusion $M_i\subset M_{i+1}$. The direct limit of highest weight
spaces is then the $\bb$-\emph{highest weight space of} $M$, where
$\ds \bb=\underrightarrow{\lim}\bb_i$.

By $$\ds\Gamma_{\gg}:\gg-\mathrm{mod} \rightsquigarrow\Int_\gg,$$
$$ M\mapsto\Gamma_{\gg}(M):=\{m\in M, \dim\span\{m, g\cdot m,
g\cdot m^2,\cdots\}<\infty \hspace{0.5 cm} \forall g\in\gg \}$$ we
denote the \textit{functor of $\gg$-integrable vectors}. It is an
exercise to check that $\Gamma_{\gg}(M)$ is indeed a well-defined
$\gg$-submodule of $M$; the fact that $\Gamma_{\gg}(M)$ is
integrable is obvious. Furthermore, $\Gamma_{\gg}$ is a left-exact
functor.

If $\gg$ is a diagonal (locally simple) Lie algebra, then one can
define a {\it natural module} $V$ of $\gg$. Indeed, the reader
will verify that one can choose a subexhaustion of \refeq{eq2}
such that the natural $\gg_i$-module $V_i$ is a $\gg_i$-submodule
of $V_{i+1}$ for any $i$. Therefore, fixing arbitrary injective
homomorphisms $V_i\to V_{i+1}$ of $\gg_i$-modules, we obtain a
direct system and we set $V:=\underrightarrow{\lim}V_i$. Note that
$V$ depends on the choice of the homomorphisms $V_i\to V_{i+1}$.
If however, $\gg\cong sl(\infty),\,o(\infty),\,sp(\infty)$, then
the homomorphisms $V_i\to V_{i+1}$ are unique up to
proportionality, and one can prove that as a result $V$ is unique
up to isomorphism, i.e. in particular does not depend on the fixed
exhaustion of $\gg$. In these latter cases we speak about
\emph{the natural representation}.

By choosing injective homomorphisms of $\gg_i$-modules $V^*_i\to
V^*_{i+1}$, we obtain a direct system defining a {\it conatural
representation} of $\gg$. We denote such a representation by
$V_*$. For $\gg\cong sl(\infty),\,o(\infty),\,sp(\infty)$ $V_*$ is
unique up to isomorphism. In fact, $V\simeq V_*$ for $\gg\cong
o(\infty),\,sp(\infty)$.

\section{Injective modules in $\Int_\gg$ and semisimplicity of the
  category $\Intwt$}

\begin{prop}\label{ext}
$\ext^1_\gg (X,M^*)=0$ for any $X,M\in \Int_\gg$.
\end{prop}
\begin{proof} We use that
$$\ext^1_\gg (X,M^*)=\ext^1_\gg (\mathbb{C},\hom_\mathbb
C(X,M^*))\simeq H^1(\gg,\hom_\mathbb
C(X,M^*))=H^1(\gg,(X\otimes M)^*),$$ see for instance [W].
Therefore it suffices to show
that $H^1(\gg,R^*)=0$ for any integrable $\gg$-module $R$.
Consider the standard complex for the cohomology of $\gg$ with
coefficients in $R^*$:
\begin{equation}
0\to R^* \to (\gg\otimes R)^* \to (\Lambda^2(\gg)\otimes R)^*\to\dots
\end{equation}
It is dual to the standard homology complex
$$0\gets R \gets \gg\otimes R \gets \Lambda^2(\gg)\otimes R \gets \dots,$$
which is the direct limit of complexes
$$0\gets R \gets \gg_i\otimes R \gets \Lambda^2(\gg_i)\otimes R \gets \dots.$$
Since $H_1(\gg_i,R)=0$ for each $i$, we get $H_1(\gg,R)=0$.
Therefore the dual complex (3) has trivial first cohomology, i.e.
$H^1(\gg,R^*)=0$.
\end{proof}

\begin{prop}\label{inj}
For any $M\in\Int_\gg$, $\Gamma_{\gg}(M^*)$ is an injective object of $\Int_\gg$.
\end{prop}
\begin{proof}
Let $X\in \Int_\gg$. The exact sequence of $\gg$-modules
$$0\to\Gamma_{\gg}(M^*)\to M^*\to M^*/\Gamma_{\gg}(M^*)\rightarrow 0$$ induces an exact sequence of vector spaces

$$
0\to\hom_{\mathbb{C}}(X,\Gamma_{\gg}(M^*))\stackrel{\phi}{\to}\hom_{\mathbb{C}}(X,M^*)\to
\hom_{\mathbb{C}} (X,M^*/\Gamma_{\gg}(M^*))\to$$
$$
\to\ext^1_{\gg}(X,\Gamma_{\gg}(M^*))\stackrel{\psi}{\to}\ext^1_{\gg}(X,M^*)=0.
$$ Since $\hom_{\mathbb{C}}(X, M^*/\Gamma_{\gg}(M^*))=0$ (this follows from the facts that a quotient of an integrable $\gg$-module is again an integrable $\gg$-module and that $\Int_{\gg}$ is closed with respect to extensions) we conclude that $\psi$ is an isomorphism, i.e. that $\ext^1_{\gg}(X,\Gamma_{\gg}(M^*))=0$. \end{proof}

 \begin{corollary}\label{hull} 
$\Int_\gg$ has enough injectives. 
\end{corollary}

\begin{proof}
Let $M\in\Int_\gg$. Then
$M\subset M^{**}$. By the very definition of $\Gamma_\gg$,
$M\subset\Gamma_\gg(M^{**})$, and $\Gamma_\gg(M^{**})$ is an injective
object of $\Int_\gg$ by Proposition 3.2.
\end{proof}

Note that there is a simpler proof of Corollary 3.3 not
referring to Proposition 3.2. Indeed it is enough to notice that
the functor $\Gamma_\gg :\gg$-mod$\rightsquigarrow \Int_\gg$ is right adjoint
to the inclusion functor $\Int_\gg\subset\gg$-mod. Then the
equality $$\hom_\gg(M,J_M)=\hom_\gg(M,\Gamma_\gg(J_M))$$
allows us to conclude that, if $i:M\rightarrow J_M$ is an
injective homomorphism of $M\in\Int_\gg$ into an injective
$\gg$-module, then $\Gamma_\gg(J_M)$ is an injective object
of $\Int_\gg$ and $i$ factors through the inclusion $\Gamma_\gg(J_M)\subset J_M$.
In particular, this argument allows to reduce the existence of
injective hulls in $\Int_\gg$ to the well-known existence of
injective hulls in $\gg$-mod.

With this in mind, we can view Propositions 3.1 and 3.2 as yielding
an explicit construction of an injective module $\Gamma_\gg(M^*)$
associated to any $M\in\Int_\gg$.

In the rest of this section we assume that $\gg$ admits a
splitting Cartan subalgebra $\hh\subset\gg$, i.e. an abelian
subalgebra $\hh\subset\gg$ such that $\gg$ decomposes as
$$\hh\oplus\bigoplus_{\alpha\in\hh^*}\gg^\alpha,$$
where $$\gg^\alpha=\{g\in \gg | [h,g]=\alpha(h)g\,\text{for
any}\,h\in\hh\}.$$ It is well-known that in this case $\gg$ is
isomorphic to a direct sum of copies of
$sl(\infty),o(\infty),sp(\infty)$ and finite-dimensional simple
Lie algebras, see [PStr].

We define the category $\Intwt$ as the full subcategory of
$\Int_\gg$ which consists of {\it weight modules} $M$, i.e.
objects $M\in\Int_\gg$ which admit a decomposition
\begin{equation}\label{eq3}M=\bigoplus_{\alpha\in\hh^*}M^\alpha,\end{equation} where
$$M^\alpha=\{m\in M | h\cdot m=\alpha(h)m\,\text{for any}\,h\in\hh\}.$$
Note that \refeq{eq3} is automatically a decomposition of
$\hh$-modules. It is also clear that there is a left exact functor
$$\Gamma^{\mathrm{wt}}_\hh:\Int_\gg\rightsquigarrow \Intwt,\,\, M\mapsto \bigoplus_{\alpha\in\hh^*}M^\alpha.$$
By $\Gamma^{\mathrm{wt}}_{\gg,\hh}$ we denote the composition
$$\Gamma^{\mathrm{wt}}_\hh\circ \Gamma_\gg:\gg\text{-mod} \rightsquigarrow \Intwt.$$

\begin{lemma}\label{wtinj} If $X$ is an injective object of
  $\Int_\gg$, then $\Gamma^{\mathrm{wt}}_\hh(X)$ is an injective object of $\Intwt$.
\end{lemma}
\begin{proof} 
It suffices to note that $\Gamma_\gg^{\mathrm{wt}}$ is a right adjoint
to the inclusion functor $\Int_{\gg,\hh}^{\mathrm{wt}}\subset\Int_\gg$.
\end{proof}
\begin{example}\label{ex36}{\rm  Let $\gg=sl(\infty)$
and $M=V\otimes V_*$ . Consider the $\gg$-module
$M^*$. Let's think of $M^*=(V\otimes V_*)^*$ as the space of all
infinite matrices $B=(b_{ij}), i,j\in \mathbb Z_{>0}$, and of $M$
as the space of finitary infinite matrices $A=(a_{ij}), i,j\in
\mathbb Z_{>0}$, where $B(A)=\sum_{i,j} b_{ij}a_{ji}$. Then $\gg$
is identified with the subspace $F\subset (V\otimes V_*)^*$ of
finitary matrices with trace zero, and the $\gg$-module structure
on $M^*$ is given by $A\cdot B=[A,B]$. We fix the Cartan
subalgebra $\hh$ to be the algebra of finitary diagonal matrices, and
we claim that $\Gamma^{\mathrm{wt}}_\hh(M^*)=F+D$ where $D$ is
the subspace of diagonal matrices. Indeed, clearly $D$ equals the
$\hh$-weight space $(M^*)^0$ of weight 0. Furthermore, any
non-zero eigenspace of $\hh$ is the span of an elementary
non-diagonal matrix, hence $\Gamma^{\mathrm{wt}}_\hh(M^*)=F+D$.
Note also that we have a non-splitting exact sequence of
$\gg$-modules
$$0\to\gg\to\Gamma^{\mathrm{wt}}_\hh(M^*)\to T\to 0,$$ where $T=D/D\cap
F$ is a trivial $\gg$-module of dimension $\beth_1$. }
\end{example}

\begin{corollary}\label{wtinj1} For any $M\in\Int_\gg$, $\Gamma^{\mathrm{wt}}_{\gg,\hh}(M^*)$
is an injective object of $\Intwt$.
\end{corollary}

Define now $\Intfin$ as the full subcategory of $\Intwt$ consisting of
$\hh$-weight modules $M=\bigoplus_{\alpha\in\hh^*}M^\alpha$ such that
$\dim M^\alpha<\infty$ for any $\alpha\in\hh^*$.

\begin{theo}\label{simpfin} The category $\Intfin$ is semisimple.
\end{theo}
\begin{proof}
Let $M\in\Intfin$ be simple. Then there is an $\hh$-module isomorphism
$$\displaystyle M=\oplus_{\alpha\in\hh^*}M^{\alpha}.$$
Therefore $M^*=\prod_{\alpha\in\hh^*} (M^{\alpha})^*$. A
non-difficult computation shows that
$\Gamma^{\mathrm{wt}}_\hh(M^*)$ is isomorphic to
$\oplus_{\alpha\in\hh^*}(M_{\alpha})^*$. Moreover, using the fact that $\dim M^\alpha<\infty$ for all $\alpha$, it is easy to
check that $\displaystyle
M_*:=\oplus_{\alpha\in\hh^*}(M_{\alpha})^*$ is a simple integrable
$\gg$-module. Hence $M_*=\Gamma^{\mathrm{wt}}_{\gg,\hh}(M^*)$.
Applying $\Gamma^{\mathrm{wt}}_{\gg,\hh}$ again, we see that
$$\Gamma^{\mathrm{wt}}_{\gg,\hh}(\Gamma^{\mathrm{wt}}_{\gg,\hh}(M^*)^*)= M.$$
Therefore $M$ is injective in $\Intwt$, and thus also in $\Intfin$, by Corollary \ref{wtinj1}.
\end{proof}

\begin{example}\label{ex38}\end{example}

a) Let $\gg=sl(\infty)$. One checks immediately that all tensor
powers $V^{\otimes k}$, $V$ being the natural module, are objects
of $\Intfin$. The same applies to the tensor powers of the
conatural module $V_*$. However, the category $\Intfin$ contains
also more interesting modules as the following one:
$M=\underrightarrow{\lim} S^i(V_i)$, $V_i$ being the natural
representation of $sl(i)$ . The module $M$ has 1-dimensional
weight spaces, but is not a highest weight module, see [DP1,
Example 3]. Note also that the adjoint representation is not an
object of $\Intfin$.

b) Let $\gg=o(\infty)$ and let $\gg$ be exhausted by
$\gg_i=o(2i),i\geq 3$. Denote by $S_i^1$ and $S_i^2$ the two
non-isomorphic spinor $\gg_i$-modules. Then $S_i^1$ and $S_i^2$
are both isomorphic to $S_{i-1}^1\oplus S_{i-1}^2$ as
$\gg_{i-1}$-modules. Therefore there is an injective homomorphism
of $\gg_{i-1}$-modules $\phi^{ks}_{i-1}: S^k_{i-1}\to S^s_i$ for
$k,s\in\{1,2\}$, and moreover $\phi^{ks}_{i-1}$ is unique up to
proportionality. Any sequence $\{t_i\}_{i\geq 3}$ of elements in
$\{1,2\}$ defines a direct system
$$S_3^{t_3}\stackrel{\phi_3^{t_3,t_4}}{\longrightarrow}S_4^{t_4}\stackrel{\phi_4^{t_4,t_5}}{\longrightarrow} S_5^{t_5}\stackrel{\phi_5^{t_5,t_6}}{\longrightarrow}\dots$$
and hence a simple $\gg$-module $S(\{t_i\})$. Using the fact that
$S(\{t_i\})$ is locally simple, it is easy to see that
$S(\{t_i\})=S(\{t'_i\})$ if and only if the  ``tails'' of the
sequence $\{t_i\}$ and $\{t'_i\}$ coincide, i.e. $t_i=t'_i$ for
large enough $i$.

The modules $S(\{t_i\})$ are weight modules with 1-dimensional
spaces for any Cartan subalgebra $\hh$ of the form
$\ds\hh=\cup_i\hh_i$ where $\hh_3\subset\hh_4\subset\dots$ are
nested Cartan subalgebras of
$\gg_3=o(6)\subset\gg_4=o(8)\subset\dots$. In particular, $S(\{t_i\})\in\Intfin$.

\section{On the integrability of $M^*$ for $M\in\Int_{\gg}$}

\begin{lemma}\label{crit} Let $M\in\Int_\gg$. Then $M^*\in\Int_\gg$ if and
  only if for any $i>0$ $\hom_{\gg_i}(N,M)\neq 0$ only for finitely
  many non-isomorphic simple $\gg_i$-modules $N$.
\end{lemma}
\begin{proof} Fix $i$.
Let $\Lambda_i$ be the set of integral dominant weights of $\gg_i$
(for some fixed Borel subalgebra $\bb_i$ of $\gg_i$ with fixed
Cartan subalgebra $\hh_i\subset\bb_i$)
  and $V_{\lambda}^i$ be the simple $\gg_i$-module with highest weight
  $\lambda$.
Denote by $\Lambda_i(M)$ the set of all $\lambda\in\Lambda_i$
  such that $\hom_{\gg_i}(V_{\lambda}^i,M)\neq 0$. Since
  $M$ is a semisimple $\gg_i$-module, we can write $M$ as
\[
M=\oplus_{\lambda\in\Lambda_i(M)} M^\lambda\otimes V_{\lambda}^i,
\]
where $M^\lambda:=\hom_{\gg_i}(V_{\lambda}^i,M)$ is a trivial
$\gg_i$-module. We have
\[
M^*=\prod_{\lambda\in\Lambda_i(M)} (V_{\lambda}^i)^*\otimes
(M^\lambda)^*.
\]

Suppose that $\Lambda_i(M)$ is finite. Then for any fixed
$g\in\gg_i$ there is a polynomial $p_\lambda(z)$ such that
$p_\lambda(g)\cdot (V_{\lambda}^i)^*=0$. Set
$p(z):=\prod_{\lambda\in\Lambda_i(M)}p_\lambda(z)$. Then
$p(g)\cdot M^*=0$. Hence $g$ acts integrably on $M^{*}$, i.e.
$M^*$ is integrable over $\gg_i$.

Now let $\Lambda_i(M)$ be infinite. Let $v_\lambda$ be a non-zero
vector of weight $-\lambda$ in $(V_{\lambda}^i)^*\otimes
(M^\lambda)^*$. One can choose $h$ in the Cartan subalgebra of
$\gg_i$ such that $\lambda(h)\neq\mu(h)$ for any
$\mu\neq\lambda\in\Lambda_i(M)$. Let
$v:=\prod_{\lambda\in\Lambda_i(M)}(v_\lambda)\in
\prod_{\lambda\in\Lambda_i(M)}
(V_{\lambda}^i)^*\otimes(M^\lambda)^*$. Then $\dim(\CC[h]\cdot
v)=\infty$, and $M^*$ is not $\gg_i$-integrable. \end{proof}

\begin{corollary}\label{tendual} Let $M,M'\in\Int_\gg$. If
  $M^*,(M')^*\in\Int_\gg$, then $(M\otimes M')^*\in \Int_\gg$ and $M^{**}\in\Int_\gg$.
\end{corollary}

\begin{prop}\label{diag} Let $\gg$ be a locally simple Lie algebra.
There exists a non-trivial module $M\in\Int_{\gg}$ such that $M^*$
  is integrable if and only if $\gg$ is diagonal.
\end{prop}
\begin{proof}
First of all, if $\gg$ is diagonal, then any natural module $\ds
V=\underrightarrow{\lim}V_n$ satisfies the finiteness condition of
Lemma \ref{crit}, hence $V^*$ is integrable.

Before we prove the other direction, note that, by passing to a
subexhaustion, we can always assume that $\gg$ is exhausted by
classical simple Lie algebras $\gg_i$ of the same type (A, B, C or
D). Let now $M\in\Int_\gg$ be a non-trivial and $M^*$ be
integrable. We will show that $\gg$ is diagonal. Since $M$
satisfies the finiteness condition of Lemma \ref{crit},
$\End_\CC\,M$ and its submodules satisfy this condition too. The 
adjoint module $\gg$ is a submodule of $\End_\CC\,M$, hence this
implies that for each $i$ the number of $\gg_i$-isotypic
components in $\gg_{i+k}$ is uniformly bounded for all $k>0$.
Since the adjoint module of $\gg_i$ is isomorphic to $(V_i\otimes
V_i^*)/\mathbb C$ in the type A case, to $S^2(V_i)$ in type C, and to
$\Lambda^2(V_i)$ in types B or D, one can easily check that for
each $i$ the number of $\gg_i$-isotypic components in $V_{i+k}$ is
also uniformly bounded by for all $k>0$. Our goal is to show that
for all sufficiently large $i$, $V_{i+1}$ restricted to $\gg_i$ is
isomorphic to a direct sum of copies of $V_i$, $V_i^*$ and
$\mathbb C$.

Let us start with the type A case. Pick an $sl(2)$-subalgebra in
$\gg_n$ for some $n$. The set of $sl(2)$-weights in $V$ is finite.
Thus we can let $k\in\ZZ_{>0}$ be the maximal weight in this set
and fix $i$ such that $k$ is a weight of $V_i$. Note that
$sl(2)\subset\gg_i$. Then we have an isomorphism of $\gg_i$-modules
$$V_{i+1}=T_{\lambda_1}(V_i)\oplus\dots\oplus T_{\lambda_s}(V_i),$$
where each $\lambda_j$ is a Young diagram and $T_{\lambda_j}(V_i)$ is
the image of the corresponding Young projector in the appropriate
tensor power of $V_i$. Since $V_{i+1}$ does not have any weight
greater than $k$, each diagram $\lambda_j$ has only one column. Indeed,
otherwise we can put a vector of weight $k$ in each box of the first row and
put other weight vectors in all other boxes of $\lambda_j$ so that
the total sum of all weights of vectors is greater than $k$, which
contradicts the fact that $k$ is the maximal weight. Next we claim
that the length of this column equals $0, 1$, $\dim V_i$, or $\dim
V_i-1$. Indeed, if we put in the boxes of $\lambda_i$ linearly
independent vectors of maximal possible sum of weights, the total
sum is not greater than $k$ only in these four cases. Hence each
simple $\gg_i$-constituent of $V_{i+1}$ is isomorphic to $V_i$,
$V_i^*$ or $\mathbb C$ (the numbers $0$ and $\dim V_i$ correspond both 
to the trivial 1-dimensional $\gg_i$-module).

If each $\gg_i$ is of type B or C, D, let $\mathfrak s_i\subset
\gg_i$ be a maximal root subalgebra of type A. Notice that by the
previous argument the restriction of $V_{i+1}$ on $\mathfrak s_i$
is a sum of natural, conatural and trivial modules. That is only
possible if the restriction of $V_{i+1}$ to $\gg_i$ is a sum of
natural and trivial modules.
\end{proof}

Proposition 4.3 follows also from Corollary 3.9 in [Ba2].

\begin{example}\end{example}

a) Let $\gg=sl(\infty)$, and let
$M=\underrightarrow{\lim}S^i(V_i)$ be as in \refex{ex38}, a). Then
$\hom_{\gg_i}(S^k(V_i),S^j(V_j))\neq 0$ for all $i$, $k\leq j$.
Hence $\hom_{\gg_i}(S^k(V_i),M)\neq 0$ for all $k>0$, and by Lemma
\ref{crit} $M^*$ is not an object of $\Int_\gg$.

b) Consider the case $\gg=o(\infty)$ and let $S(\{t_i\})$ be the
$\gg$-module defined in \refex{ex38}, b). Then if $N$ is a simple
$\gg_i$-module, $\hom_{\gg_i}(N,S(\{t_i\}))\neq 0$ iff $N\simeq
S^1_i$ or $N\simeq S^2_i$. Hence $S(\{t_i\})^*\in \Int_\gg$ by
Lemma \ref{crit}. Moreover, $S(\{t_i\})^*$  is injective by
Proposition \ref{inj}.

c) Let $\gg=sl(\infty)$ and let $M$ be as in Example 3.5. Then
$\hom_{\gg_i}(N,M)\neq 0$ if $N$ is isomorphic to one of the
following simple $\gg_i$-modules: trivial, natural, conatural, adjoint.
Therefore $M^*$ is $\gg$-integrable and injective in $\Int_\gg$.
Furthermore, $M^*\cong\CC\oplus \gg^*$.

\section{On the Loewy length of $\Gamma_{\gg}(M^*)$ for $M\in\Int_{\gg}$}

Recall that the  {\it socle}, $\soc (M)$, of a $\gg$-module $M$ is
the largest semisimple submodule of $M$. The {\it socle
filtration} of $M$ is the filtration of $\gg$-modules
$$0\subset \soc  (M)\subset \soc^1 (M)\subset\dots\subset \soc^i (M)\subset\dots,$$
where $\soc^i (M)=p_i^{-1}(\soc (M/\soc^{i-1}(M))$ and $p_i:M\to
M/\soc^{i-1} (M)$ is the natural projection. We say the the socle
filtration of $M$ is {\it exhaustive} if $M=\underrightarrow{\lim}
(\soc^i(M))$. We say that $M$ has {\it finite Loewy length} if the
socle filtration of $M$ is finite and exhaustive. The {\it Loewy
length} of $M$ equals $k+1$ where $k=\min\{r\ |\ \soc^r(M)=M\}$.

\begin{prop}\label{mult} Let $M\in\Int_{\gg}$ be a simple $\gg$-module such that $\Gamma_{\gg}(M^*)$
  has finite Loewy length. Then there exist $n\in\ZZ_{>0}$ and a direct system $M_i$ of
  simple finite-dimensional $\gg_i$-modules such that  $M=\underrightarrow{\lim} M_i$
  and $\dim\hom_{\gg_i} (M_i,M_{j})=1$ for
  all $j>i>n$.
\end{prop}
We first prove several lemmas.
\begin{lemma}\label{extension} Let $Q=\underrightarrow{\lim} Q_i\in\Int_\gg$, where $Q_i$ are finite-dimensional, not necessarily simple, $\gg_i$-modules. Assume that for all sufficiently
large $i$ there exists a simple $\gg_i$-submodule $X_i\subset Q_i$ such that
$\dim\hom_{\gg_i}(X_i,X_{i+1})>2$. Then there exists a locally simple
module $X=\underrightarrow{\lim} X_i\in\Int_{\gg}$ and a non-trivial extension of $\gg$-modules
$$0\to Q\to Z\to X\to 0.$$
\end{lemma}
\begin{proof} Fix a sequence of injective
homomorphisms of $\gg_i$-modules $f_i:X_i\to X_{i+1}$ and set
$X=\underrightarrow{\lim}\,X_i $. Let $Z_i:=X_i\oplus Q_i$ and
consider the injective homomorphisms of $\gg_i$-modules
$$\ds a_i:Z_i\to Z_{i+1},\ a_i((x, q)):=(f_i(x), t_i(x)+e_i(q)),$$
where $t_i$ are some injective homomorphisms $X_i{\to}Q_{i+1}$,
$e_i:Q_i\to Q_{i+1}$ are the given inclusions, and $q\in Q_i$,
$x\in X_i$. Put $\ds Z:=\underrightarrow{\lim} Z_i$.

Then, clearly, $Q$ is a submodule of $Z$ and the quotient $Z/Q$ is
isomorphic to $X$. Thus we have constructed an extension of $X$ by
$Q$. This extension splits if and only if for all sufficiently
large $i$ there exist non-zero homomorphisms $p_i:X_i\to Q_i$ such
that $t_i=p_{i+1}\circ f_i-e_i\circ p_i$, see the following diagram:

$$\begin{array}[c]{ccc}
X_{i+1}&\stackrel{p_{i+1}}{\rightarrow}&Q_{i+1}\\
\uparrow\scriptstyle{f_i}&\scriptstyle{t_i}\nearrow &\uparrow\scriptstyle{e_i}\\
X_{i}&\stackrel{p_i}{\rightarrow}&Q_{i}.
\end{array}$$

 Assume that for any
choice of $\{t_i\}$ such a splitting exists. If
$n_i:=\dim\hom_{\gg_i}(X_i,Q_i)$, this assumption implies
$$\dim \hom_{\gg_i}(X_i,Q_{i+1})\leq n_i+n_{i+1}.$$
On the other hand, $\dim \hom_{\gg_i}(X_i,Q_{i+1})\geq k_i
n_{i+1}$ where $k_i:=\dim\hom_{\gg_i}(X_i,X_{i+1})$. Since
$k_i>2$, we have $n_{i+1}< n_i$.  As $n_i>0$ for all $i$, we
obtain a contradiction.
\end{proof}
\begin{corollary}\label{ex} Let $Q\in\Int_\gg$ be a simple $\gg$-module
satisfying the assumption of Lemma \ref{extension}. Then $Q$ admits no
non-zero homomorphism into an injective object of $\Int_\gg$ of
finite Loewy length.
\end{corollary}
\begin{proof}For any $m>0$ we will now construct
an integrable module $Z^{(m)}\supset Q$ whose socle equals $Q$ and whose Loewy
length is greater than $m$. For $m=1$ this was done in Lemma
\ref{extension}. Proceeding by induction, we set
$$Z_i^{(m)}:=X_i\oplus Z_i^{(m-1)}=X_i\oplus(X_i\oplus Z_i^{(m-2)})$$
and define $a_i^{(m)}:Z_i^{(m)}\to Z_{i+1}^{(m)}$ by
$$a_i^{(m)}(x,x',z)=(f_i(x),r_i^{(m-1)}(x)+f_i(x'),t_i^{(m-2)}(x')+q_i^{(m-2)}(z)),$$
where now $\{t_i^{(m-2)}\}$ is a set of non-zero homomorphisms
$t_i^{(m-2)}:X_i\to Z_{i+1}^{(m-2)}$ and $\{r_i^{(m-1)}\}$ is a
set of non-zero homomorphisms $r_i^{(m-1)}:X_i\to X_{i+1}$. As in
the proof of Lemma \ref{extension} one can choose
$\{t_i^{(m-2)}\}$ and $\{r_i^{(m-1)}\}$ so that $Z^{(m)}$ is a
non-split extension of $X$ by $Z^{(m-1)}$, and $Z^{(m)}/Z^{(m-2)}$
is a non-split self-extension of $X$. Therefore the Loewy length
of $Z^{(m)}$ is greater than $m$. The statement follows. 
\end{proof}
\begin{lemma}\label{locally} Let $Q=\underrightarrow{\lim} Q_i\in\Int_\gg$ 
be a simple $\gg$-module which admits a non-zero homomorphism into an 
injective object of $\Int_\gg$ of finite
Loewy length. Then there exist $n\in\ZZ_{>0}$ and a direct system
of simple $\gg_i$-submodules $S_i$ of $Q$ such that
$Q=\underrightarrow{\lim} S_i$ and $\dim\hom_{\gg_i}
(S_i,S_{j})=1$ for all $j>i>n$.
\end{lemma}
\begin{proof} Decompose each $Q_i$ into a
direct sum of isotypic components, $Q_i=Q_i^1\oplus\dots\oplus Q_i^{l(i)}$. 
We define a directed graph $\Gamma$  as follows. The set of
vertices $V(\Gamma)$ is by definition $\{Q_i^j\}$, and
$V(\Gamma)=\cup_{i>0}V(\Gamma)_i$, where
$V(\Gamma)_i=\{Q_i^1,\dots,Q_i^{l(i)}\}$. An edge $A\to B$ belongs
to $\Gamma$ if $A\in V(\Gamma_i)$, $B\in V(\Gamma_{i+1})$ and
$\hom_{\gg_i} (A,B)\neq 0$.

Let $\Gamma_{>i}$ be the full subgraph of $\Gamma$ whose set of
vertices equals $\cup_{k>i}V(\Gamma)_k$. For any vertex $A$ of
$\Gamma$ we denote by $V(A)$ the set of vertices $B$ such that
there is a directed path from $A$ to $B$. Let $\Gamma(A)$ be the full
subgraph of $\Gamma$ whose set of vertices equals $V(A)$, and
$\Gamma(A)_{>i}$ be the full subgraph of $\Gamma(A)$ whose set of
vertices equals $\cup_{k>i}(V(\Gamma)_k\cap V(A))$. Note that the
simplicity of $Q$ implies that $\Gamma_{>i}$ and $\Gamma(A)_{>i}$
are connected. In particular, if $\Gamma(A)$ is a tree, then
$\Gamma(A)$ is just a string.

We will now prove that there exists a vertex $A$ such that
$\Gamma(A)$ is a tree. Indeed, assume the contrary. This implies
that one can find an infinite sequence of vertices $A_1\in
V(\Gamma)_{i_1},A_2\in V(\Gamma)_{i_2},\dots$ such that the number
of paths from $A_n$ to $A_{n+1}$ is greater than 2 for all $n$.
Then $Q=\underrightarrow{\lim} Q_{i_k}$. In addition, one can
easily see that $Q$ satisfies the assumption of Lemma
\ref{extension} and hence $Q$ admits no non-zero homomorphism into
an injective object of $\Int_\gg$ of finite Loewy length.
Contradiction.

Fix now $A\in V(\Gamma)_i$ such that $\Gamma(A)$ is a tree. Then,
as we mentioned above, $V(\Gamma)$ is necessarily a string
$A_i=\{A\to A_{i+1}\to A_{i+2}\dots\}$. Let $S_j$ be a simple
submodule of $A_j$, $j\geq i$. Then by Lemma \ref{extension} there
exists $n$, such that $\dim\hom_{\gg_j}(S_j,S_k)=1$ for any
$k>j\geq n$. Fix $s\in S_n$ and set $S_j=U(\gg_j)\cdot s$ for all
$j\geq n$. Then $S_j$ are simple and $Q=\underrightarrow{\lim}
S_j$ satisfies the condition in the lemma. 
\end{proof}
\begin{lemma}\label{unique} Let $Q=\underrightarrow{\lim} S_i\in\Int_\gg$,
where $S_i$ are simple $\gg_i$-modules such that, for some
$n$, $\dim\hom_{\gg_i}(S_i,S_{j})=1$ for all $j>i>n$. Then $Q^*$ has a
unique simple submodule $Q_*$, and $Q_*\in\Int_\gg$.
\end{lemma}
\begin{proof} The condition on $Q$ implies
that $\dim\hom_{\gg_i}(S_i,Q)=1$ for all sufficiently large $i$.
Therefore $\dim\hom_{\gg_i}(S^*_i,Q^*)=1$ for all sufficiently
large $i$. Note also that $Q_*=\underrightarrow{\lim} S^*_i$ is
uniquely defined (as $\dim\hom_{\gg_i}(S_i,S_{i+1})=1$) and is a
simple integrable submodule of $Q^*$. Let $S$ be some simple submodule of
$Q^*$. Since $\ds Q^*=\lim_{\gets}S_i^*$ and
$\hom_{\gg}(S,Q^*)\neq 0$, we have $\hom_{\gg_i}(S,S_i^*)\neq 0$
for some $i$. Therefore $S_i^*\subset S$ as the multiplicity of
$S^*_i$ in $Q^*$ is 1. This implies $S=Q_*$.
\end{proof}

We are now ready to prove \refprop{mult}. \\{\it Proof of Proposition} \ref{mult}. Fix $0\neq m\in M$ and put $M_i:=U(\gg_i)\cdot
m$. Then, by the simplicity of $M$, we have
$M=\underrightarrow{\lim} M_i$. Since $\Gamma_\gg(M^*)$ has finite
Loewy length, $M^*$ has a simple submodule $Q$. By
\refle{locally}, $Q$ satisfies the assumption of \refle{unique}.
The composition of the canonical injection $M\to (M^*)^*$ and the
dual map $(M^*)^*\to Q^*$ defines an injective homomorphism $M\to Q^*$. By
\refle{unique} $M\simeq Q_*$ and, since $Q_*$ also satisfies the assumption
of Lemma \ref{unique}, we conclude that the claim of Proposition
\ref{mult} holds for $M$.$\Box$

The following statement is a direct consequence of Proposition \ref{mult}.

\begin{corollary}\label{ffin}Let $M\in\Int_\gg$ be a simple $\gg$-module such that $\Gamma_{\gg}(M^*)$
  has finite Loewy length. Then for any sufficiently large $i$
  there exists a simple $\gg_i$-module $N$ such that $\dim\hom_{\gg_i}
  (N,M)=1$.
\end{corollary}
The next corollary is a direct consequence of \refle{unique} and Proposition \ref{mult}.

\begin{corollary}\label{simplesoc}Let $M\in\Int_\gg$ be a simple $\gg$-module such that $\Gamma_{\gg}(M^*)$
  has finite Loewy length. Then $M^*$ has a unique simple submodule $M_*$, and $M_*\in\Int_\gg$.
\end{corollary}

\begin{theo}\label{fin} Let $\gg$ be a locally simple algebra which has a
  non-trivial module $M$ such that $M^*$ is integrable and has finite
  Loewy length, then $\gg$ is isomorphic to $sl(\infty)$, $o(\infty)$
  or $sp(\infty)$.
\end{theo}
\begin{proof} By Proposition \ref{diag} we know that $\gg$ is diagonal.
Assume that $\gg$ is not finitary and there exists $M$ satisfying
the conditions of the theorem. Also assume that in the restriction
of $V_i$ to $\gg_{i-1}$ there is no costandard module (for types
B, C and D it is automatic). Let
$\gg=\underrightarrow{\lim}\gg_i$. Fix $n$ and let
$\phi_k:\gg_n\to \gg_{n+k}$ denote the inclusion defined by our
fixed exhaustion of $\gg$. Since $\gg$ is diagonal, there exists a
root subalgebra $\ll_k\subset\gg_{n+k}$ such that $\ll_k\simeq
\gg_n\oplus\dots\oplus \gg_n$ and $\phi_k(\gg_n)$ is the diagonal
subalgebra in $\ll_k$. Let $a_k$ be the number of simple direct
summands in $\ll_k$. Since $\gg$ is not finitary, $a_k\to\infty$.

By Corollary \ref{ffin} $M=\underrightarrow{\lim}M_i$ is a direct
limit of simple modules and, by possibly increasing $n$, we have
$\dim\hom_{\gg_n}(M_n,M_{n+k})=1$ for all $k$. Choose a set of
Borel subalgebras $\mathfrak b_i\subset \gg_i$ such that
$\phi_k(\mathfrak b_n)\subset \mathfrak b_{n+k}$. Let $h$ be the
highest coroot of $\gg_n$ and let $\lambda$ be the highest weight
of some simple $\ll_k$-constituent $L$ of $M_{n+k}$. Since $M^*$
is integrable, Lemma \ref{crit} implies that $\lambda(\phi_k(h))$
is bounded by some number $t$. If $h_1,\dots,h_{a_k}$ are the
images of $\phi_k(h)$ in the simple direct summands of $\ll_k$
under the natural projections, we have $\lambda(h_j)\neq 0$ for at
most $t$ direct summands. Therefore $L$ isomorphic to an outer
tensor product of at most $t$ non-trivial simple $\gg_n$-modules.
Since $M_{n+k}$ is invariant under permutation of direct summands
of $\ll_k$, we have at least $a_k-t$ simple constituents of
$M_{n+k}$ obtained from $L$ by permutation of the simple direct
summands of $\ll_k$. Note that all these simple constituents are
isomorphic as $\phi_k(\gg_n)$-modules. Thus the multiplicity of
any simple $\phi_{n+k}(\gg_n)$-module in $M_{n+k}$ is at least
$a_k-t$. Since $a_k\to\infty$, this contradicts Proposition
\ref{mult}.

The case when the restriction of $V_n$ to $\gg_{n-1}$ contains a
costandard simple constituent can be handled by a similar argument
which we leave to the reader.
\end{proof}

\section{The category $\lind_\gg$ for $\gg\simeq sl(\infty), o(\infty), sp(\infty)$}

Define $\lind_\gg$ as the largest full subcategory of $\Int_\gg$
which is closed under algebraic dualization and such that every
object in it has finite Loewy length.

It is clear that $\lind_\gg$ is closed with respect to finite direct
sums, however $\lind_\gg$ is not closed with respect to arbitrary
direct sums (see Corollary 6.17 below). Note also that, if $\gg$
is finite-dimensional and semisimple, the objects of $\lind_\gg$
are integrable modules which have finitely many isotypic components.

It follows from Theorem \ref{fin} that if $\gg$ is locally simple
and $\lind_\gg$ contains a non-trivial module, then $\gg$ is
finitary. In the rest of this section we assume that $\gg\simeq sl(\infty),\,
o(\infty)$ or $sp(\infty)$.

Set $T^{p,q}:=V^{\otimes p}\otimes (V_*)^{\otimes q}$, where $V$ and
$V_*$ are respectively the natural and conatural $\gg$-modules
($V_*\simeq V$ when $\gg\simeq o(\infty), sp(\infty)$).
The modules $T^{p,q}$ have been studied in [PS]; in particular,
$T^{p,q}$ has finite length and is semisimple only if $pq=0$ for
  $\gg=sl(\infty)$, and if $p+q\leq 1$ for $\gg=o(\infty), sp(\infty)$.
Moreover, the Loewy length of
  $T^{p,q}$ equals $\min\{p,q\}+1$ for $\gg=sl(\infty)$ and
  $[\frac{p+q}{2}]+1$ for $\gg=o(\infty),sp(\infty)$.
A simple module $M$ is called a {\it simple tensor module} if it is a
  submodule (or, equivalently, a subquotient) of $T^{p,q}$
 for some $p,q$.

It is well-known that there is a choice of nested Borel subalgebras
$\bb_i\subset\gg_i$ such that all simple tensor modules are $\bb$-highest
weight modules for $\bb=\underrightarrow{\lim}\bb_i$, see [PS]. (Moreover,
the positive roots of any such $\bb$ are not generated by the simple roots
of $\bb$. However, in the present paper we will make no further reference
to this fact.)

Denote by $\Theta$ the set of all highest weights of simple tensor
 modules. If $\lambda\in \Theta$, by $V_\lambda$ we denote the simple tensor module
 with highest weight $\lambda$, and, as in section 4, by $V_\lambda^i$ we denote the simple $\gg_i$-highest weight module with highest weight $\lambda$
 (here $\lambda$ is considered as a weight of $\gg_i$). It is easy to check (cf [PS])
that every $\lambda\in \Theta$ can be written in the form
 $\lambda=\sum a_i\gamma_i$ for some finite set
 $\gamma_1,...,\gamma_s$ of linearly independent weights of $V$ and some $a_i\in\mathbb{Z}$.
We put $|\lambda|:=\sum |a_i|$.
 It is not hard to see that for any $k$
the set of all $|\mu|\leq k$ in $\Theta$ is
 finite.
It follows from [PS] that all simple subquotients of
$T^{p,q}$ are isomorphic to $V_{\mu}$ with
$|\mu| \leq p+q$, and that if
 $V_\lambda$ is a submodule in $T^{p,q}$ then $|\lambda|=p+q$.

Note that $(T^{p,q})^*,(T^{p,q})^{**}$, etc., are integrable
modules. Indeed, it is easy to see (cf. [PS]) that for any fixed
$\lambda$ and any fixed $i>0$ the non-vanishing of $\hom_{\gg_i}(N,V_\lambda)$
for a simple $\gg_i$-module $N$ implies
$N\simeq V_\mu^i$ for $|\mu|\leq|\lambda|$. Hence the condition of
\refle{crit} is satisfied for $T^{p,q}$ for fixed $p,q$. This
shows that $(T^{p,q})^*\in\Int_\gg$. By Corollary \ref{tendual},
 $(T^{p,q})^{**}\in\Int_\gg$, etc..
\begin{lemma}\label{tp} Fix $p,q\in\ZZ_{\geq 0}$.

a) $(T^{p,q})^*$ has finite Loewy length, and all simple
subquotients of $(T^{p,q})^*$ are tensor modules of the form
$V_\lambda$ for $|\lambda|\leq p+q$. 

b) The direct product $\ds\prod_{f\in\mathcal F} T^{p,q}_f$ of any
family $\mathcal F=\{T^{p,q}_f\}$ of copies of $T^{p,q}$ has
finite Loewy length, and all simple subquotients of
$\ds\prod_{f\in\mathcal F} T^{p,q}_f$ are tensor modules  of the form
$V_\lambda$ for $|\lambda|\leq p+q$.
\end{lemma}
\begin{proof} First we prove b) using induction in $p+q$. The case $p+q=0$ is trivial.
If $p+q>0$, without loss of generality we can assume that $p>0$
(if $p=0$ and $q>0$ we replace $V$ by $V_*$ in the argument
below). There is a canonical injective homomorphism $\ds
U\to\prod_{f\in\mathcal F} T^{p,q}_f$, where $\ds U:=V\otimes
\prod_{f\in\mathcal F} T^{p-1,q}_f$, so we can consider $U$ as a
submodule of $\ds\prod_{f\in\mathcal F} T^{p,q}_f$. By the
induction assumption b) holds for $\ds\prod_{f\in\mathcal F}T^{p-1,q}_f$. 
 Since
$T^{r,s}$ has finite length for all $r,s$, [PS], this implies that 
$U$ has finite Loewy length and all simple subquotients of $U$ are
simple tensor modules  of the form
$V_\lambda$ for $|\lambda|\leq p+q$. The quotient $\ds(\prod_{f\in\mathcal F}
T^{p,q}_f)/U$ is isomorphic to a submodule of $\ds
R:=\prod_{f\in\mathcal F} (V'\otimes T^{p-1,q}_f)$, where $V'$ is
a copy of the vector space $V$ with trivial $\gg$-module
structure. Since $\ds R\simeq \prod_{f\in\mathcal F}(\bigoplus_{i\in\ZZ}T^{p-1,q}_{f,i})$, 
by the induction assumption b)
holds for  $R$. Therefore b) holds for
$\ds\prod_{f\in\mathcal F} T_f^{p,q}$.

a) To prove that $(T^{p,q})^*$ has finite Loewy length, we consider
$U':=V_*\otimes (T^{p-1,q})^*$ as a submodule of $(T^{p,q})^*$. By
the induction assumption, $U'$ has finite Loewy length. The
quotient $(T^{p,q})^*/U'$ is a submodule of
$R'=\ds\prod_{i\in\ZZ}(T^{p-1,q}_i)^*$. The latter $\gg$-module
has finite Loewy length by induction assumption and b). The
statement about the simple subquotients of $(T^{p,q})^*$ 
follows by an induction argument similar to the one in the
proof of b). This proves a)
for $(T^{p,q})^*$.
\end{proof}

\begin{example}\end{example}

a) We start with the simplest example. Let $\gg=sl(\infty),
o(\infty), sp(\infty)$ and $M=V^*=(T^{1,0})^*$. Then
$M\in\lind_\gg$ by Lemma 6.1. Furthermore, $M$ is an
injective object of $\Int_\gg$ by Proposition \ref{inj}. It is
easy to see that $\soc(M)=V_*$ and that $M/\soc(M)=V^*/V_*$ is a
trivial module of cardinality $\beth_1$. Since $\soc(M)$ is simple, $M$ is
an injective hull of $V_*$.

b) Let $\gg$ be as in a) but let $M=V^{**}=(T^{1,0})^{**}$. The exact
sequence $0\to V_*\to V^*\to V^*/V_*\to 0$ yields an exact
sequence
\begin{equation}\label{eq4}0\to (V^*/V_*)^*\to M\to (V_*)^*\to
0.\end{equation} Since $(V^*/V_*)^*$ is a trivial $\gg$-module
(cf. a)), it is injective, and hence \refeq{eq4} splits. This
yields an isomorphism $M=V^{**}=(V_*)^*\oplus T$, $T$ being a
trivial $\gg$-module of cardinality $\beth_2$.

c) Here is a more interesting example. We consider the
$\gg$-module $M^*$ where $\gg=\sl(\infty)$ and $M=V\otimes
V_*=T^{1,1}$ as in Example \ref{ex36}. Recall the
notation introduced in \refex{ex36}. In addition, let $Sc$ be the
one-dimensional space of scalar matrices, and $F_r$ (respectively
$F_c$) denote respectively the spaces of matrices with finitely
many non-zero rows (resp., columns) ($F$ has codimension 1 in
$F_r\cap F_c$). It is important to notice that $\gg\cdot
M^*\subset F_r+F_c$.

We first show that $\soc (M^*)=Sc\oplus F=\mathbb C\oplus \gg$. It
is obvious that $Sc\oplus F\subset \soc(M^*)$. To see that
$Sc\oplus F=\soc(M^*)$, let $X$ be any non-trivial simple
submodule of $\soc(M^*)$ not lying in $Sc\oplus F$. Consider
$0\neq x\in X$. Then $\gg\cdot x\subset F_r+F_c$. Furthermore, it
is easy to check that for any $0\neq y\in F_r+F_c$, there exists
$A\in\gg$ such that $A\cdot y\in F$ and  $A\cdot y\neq 0$. Hence $X=F$. Since it
is clear that $Sc$ is the largest trivial $\gg$-submodule of
$M^*$, we have shown that $\soc(M^*)=Sc\oplus F$.

We now compute $\soc^1(M^*)$. We claim that
$F_r+F_c\subset\soc^1(M^*)$. Since $BA\in F$ for $B\in F_r$, $A\in
F$, the action of $\gg$ on $F_r/F$ is simply left multiplication.
Using this it is not difficult to establish an isomorphism of
$\gg$-modules $F_r/F\simeq\bigoplus_{q\in Q}V_q$, where $Q$ is a
family of copies of $V$ of cardinality $2^\ZZ$. Similarly,
$F_c/F\simeq\bigoplus_{q\in Q}(V_*)_q$. (It is convenient to think
here of $V_*$ as the space of all row vectors each of which have finitely
many non-zero entries.) This implies $F_r+F_c\subset\soc^1 (M^*)$.

On the other hand $M^*/(F_r+F_c)$ is a trivial $\gg$-module as
$\gg\cdot M^*\subset F_r+F_c$. In order to compute $\soc^1(M^*)$ we
need to find all $z\in M^*$ such that $\gg\cdot z\subset Sc+F$. A
direct computation shows that $\gg\cdot z\in Sc+F$ if and only $z\in
J$, $J$ denoting the set of matrices each row and each column of
which have finitely many non-zero elements. (In fact, $g\cdot J \subset F$). Thus 
$\soc^1(M^*)=F_r+F_c+J$, and we obtain
the socle filtration of $M^*$:
$$0\subset Sc\oplus F\subset F_r+F_c+J\subset M^*.$$
In particular, the Loewy length of $M^*$ equals 3, the irreducible
subquotients of $M^*$ up to isomorphism are $\mathbb C,V,V_*,\gg$,
and all of them occur with multiplicity $2^\ZZ$, except $\gg$
which occurs with multiplicity 1.

Note that $M^*$ is decomposable and is isomorphic to
$\CC\oplus \gg^*$. As the socle
of $\gg^*$ is simple (being isomorphic to $\gg$), $\gg^*$ is
indecomposable. Moreover $\gg^*$ is an injective hull of $F=\gg$.

d) We now give an example illustrating statement b) of Lemma
\ref{tp}. Let $\gg=sl(\infty), o(\infty), sp(\infty)$ and $\ds
M=\prod_{f\in\mathcal F}V_f$, $\mathcal F$ being an infinite
family of copies of the natural module $V$. Set
$M^{\fin}=\{\psi:\mathcal F\to V | \dim(\psi(\mathcal
F))<\infty\}$. Then $M^{\fin}$ is a $\gg$-submodule of $M$, and
$\gg\cdot M\subset M^{\fin}$. Hence $M/M^{\fin}$ is a trivial
$\gg$-module. Moreover, $\ds
M^{\fin}\simeq \bigoplus_{g\in 2^{\mathcal F}}V_g$, where
$2^{\mathcal F}$ is the set of subsets of $\mathcal F$. Indeed,
$M^{\fin}=\underrightarrow{\lim}(\prod_{f\in\mathcal F}(V^i)_f)=\underrightarrow{\lim}
((\prod_{f\in\mathcal F}\mathbb{C}_f)\otimes V^i)\cong
\underrightarrow{\lim}\bigoplus_{g\in 2^{\mathcal F}}(\mathbb{C}_g\otimes V^i)=
\underrightarrow{\lim}(\bigotimes_{g\in 2^{\mathcal F}}(V^i)_g)=
\bigoplus_{g\in 2^{\mathcal F}}V_g$. 

This yields an exact sequence
\begin{equation}\label{eq5}
0\to \bigoplus_{g\in 2^{\mathcal F}}V_{g} \to M\to T\to 0,
\end{equation}
$T$ being trivial module of dimension $\card\,2^{\mathcal F}$.
Since $M$ has no non-zero trivial submodules, \refeq{eq5} is in
fact the socle filtration of $M$. Consequently the Loewy length of
$M$ equals 2.

\begin{corollary}\label{tens} Let $M\in\Int_\gg$ have finite Loewy
  length and all simple subquotients of $M$ be isomorphic to
  $V_\lambda$ where $|\lambda|$ is less or equal than a fixed
  $k\in\mathbb Z_{>0}$. Then

a) for any family $\mathcal F$ $\prod_{f\in\mathcal F} M_f$ has finite Loewy length and all simple
subquotients of $\prod_{f\in\mathcal F} M_f$  are isomorphic to
  $V_\lambda$ with $|\lambda|\leq k$;

b) $M^*$ has finite Loewy length and  all simple
subquotients of $M^*$  are isomorphic to
  $V_\lambda$ with $|\lambda|\leq k$;

c) $M\in \lind_\gg$.

\end{corollary}
\begin{proof}
a) The socle filtration of $M$ induces a finite filtration on  $\prod_{f\in\mathcal F} M_f$
$$ 0\subset\prod_{f\in\mathcal F} \soc(M_f)\subset \cdots\subset
\prod_{f\in\mathcal F} \soc^i(M_f)\subset\cdots\subset
\prod_{f\in\mathcal F} M_f.$$
Furthermore,
\begin{equation}\label{filt}
\soc^i(M)/\soc^{i-1}(M)\simeq \bigoplus_{|\lambda|\leq
  k}\bigoplus_{g\in\mathcal F_{\lambda}} (V_\lambda)_g
\end{equation}
for some families $\mathcal F_{\lambda}$. Hence
$$\prod_{f\in\mathcal F}(\soc^i(M_f)/\soc^{i-1}(M_f))\simeq\bigoplus_{|\lambda|\leq  k}
\prod_{f\in\mathcal F}(\bigoplus_{g\in\mathcal
  F_\lambda}(V_\lambda)_g)_f.$$

Note that for each $\lambda$
$$\prod_{f\in\mathcal F}(\bigoplus_{g\in\mathcal
  F_\lambda}(V_\lambda)_g)_f\subset \prod_{(f,g)\in\mathcal
  F\times\mathcal F_\lambda}(V_\lambda)_{(f,g)}.$$
By \refle{tp} b), $\prod_{(f,g)\in\mathcal F\times\mathcal F_\lambda}(V_\lambda)_{(f,g)}$ 
has finite Loewy length and all its simple subquotients are isomorphic
ot $V_\mu$ with $|\mu|\leq |\lambda|\leq k$. The same holds for
$\prod_{f\in\mathcal F} (\soc^i(M_f)/\soc^{i-1}(M_f))$.
Therefore a) holds. 

b) Since all $V_\lambda$ with $|\lambda|\leq k$ satisfy the conditions
of \refle{crit}, $M$ satisfies the condition of \refle{crit} and
therefore $M^*\in\Int_\gg$.

The socle filtration of $M$ induces a finite filtration on $M^*$
$$\cdots\subset(\soc^i(M))^*\subset(\soc^{i-1}(M))^*\subset\cdots.$$
Using (\ref{filt}) we get
$$(\soc^{i-1}(M))^*/(\soc^{i}(M))^*\simeq\bigoplus_{|\lambda|\leq  k}
\prod_{g\in\mathcal F_\lambda}(V_\lambda^*)_g.$$
By \refle{tp} b) $V_\lambda^*$ has finite Loewy length and its simple
subquotients are isomorphic to $V_\mu$ with $|\mu|\leq |\lambda|$,
hence by a) the same holds for $\prod_{g\in\mathcal
  F_\lambda}(V_\lambda^*)_g$. 
This implies that b) holds.

c) Note that if $M$ satisfies the assumptions of the corollary, then
$M^*$ and all higher duals $M^{**}$ etc, satisfy the the assumptions of
the corollary. Hence $M\in\lind_\gg$.
\end{proof}

Remarkably, there is following abstract characterization of simple tensor modules.
\begin{theo}\label{slind} If $M\in\Int_\gg$ is simple and
$\Gamma_{\gg}(M^*)$ has finite Loewy length, then $M$ is a simple tensor
module.
\end{theo}
\begin{proof} By Proposition \ref{mult}, $M=\ds\lim_{\to} M_i$ for some $n\in\mathbb{Z}_+$
and simple nested $\gg_i$-submodules $M_i\subset M$ with
$\dim\hom_{\gg_i}(M_i,M)=1$ for all $i\geq n$. If $\gg=sl(\infty)$, it
is useful to consider $M$ as a $gl(\infty)$-module by extending the 
$sl(i)$-module structure on $M_i$ to a $gl(i)$-module structure in a way
compatible with the injections $M_i \rightarrow M_{i+1}$. It is easy
to see that the condition $\dim\hom_{\gg_i}(M_i,M)=1$ for all $i\geq n$ 
ensures the existence of such an extension.  Note, furthermore, that 
$\dim\hom_{gl(i)}(M_i,M)=1$. This allows us to assume that 
$\gg=gl(\infty)$ and $\gg_i=gl(i)$. 

Let now
$\cc$ denote the derived subalgebra of the
centralizer of $\gg_n$  in $\gg$. Then obviously $\cc$ is a simple
finitary Lie algebra 
whose action on $M$ induces a trivial action on $M_n$.
Hence, as a $\cc$-module, $M$ is isomorphic to a
quotient of $U(\gg)\otimes_{U(\cc\oplus\gg_n)}M_n$, or 
equivalently to a quotient of $S^{.}(\gg/(\cc\oplus\gg_n))\otimes M_n$.
Note that $\gg/(\cc\oplus\gg_n)$, considered as a $\cc$-module has
finite length and that its simple subquotients are
natural, conatural, and
possibly 1-dimensional trivial $\cc$-modules. 
This implies that every
simple $\cc$-subquotient of $M$ is a simple tensor $\cc$-module. In addition, for $i\geq n$,
the number of non-zero marks of the highest weight of any simple $\gg_i$-submodule of $M$
is not greater than $n$ plus the multiplicity of the non-trivial simple constituents of the 
$\gg_n$-module $\gg/(\cc\oplus\gg_n)$. In particular, if
$\lambda_i$ denotes the highest weight of $M_i$ then $\lambda_i$ has
at most $3n$ non-zero marks.

Consider first the case when $\gg=gl(\infty)$. Then every weight $\lambda_i$ can be
written uniquely in the form 
$$a^i_1\varepsilon_1+\dots+a^i_k\varepsilon_k+b^i_1\varepsilon_{n-k}+\dots+b^i_k
\varepsilon_n$$
for some fixed $k$, $a^i_1\geq a^i_2\geq\dots \geq a^i_k\geq 0$ and $0\geq
b^i_1\geq\dots\geq b^i_k$. 
We claim that for sufficiently large $i$ the
weight stabilizes, i.e. $a^i_j=a^{i+1}_j=\dots=a^p_j=\dots$ and
$b^i_j=b^{i+1}_j=\dots=b^p_j=\dots$ for all $j$, $1\leq j \leq k$.
Indeed, assume the contrary. Let $j$ be the
smallest index such that the sequence $\{a^i_j\}$ does not
stabilize. By the branching rule for $gl(m)\subset gl(m+1)$ (see for instance [GW]) the
sequence $\{a^i_j\}$ is non-decreasing. Hence there is $p$ such that
$a^{p+1}_j>a^p_j$. Set $\mu=\lambda_{p}+\varepsilon_j$. Then the
multiplicity of $M_{p-1}$ in $V_\mu^p$ is not zero and the multiplicity of $V_\mu^p$
in $M_{p+1}$ is not zero. Since $V_\mu^p\not=M_p$, this shows that the multiplicity of $M_{p-1}$ in
$M_{p+1}$ is at least 2. Contradiction. Similarly the sequence $\{b^i_j\}$
stabilizes. As it is easy to see, this is sufficient to conclude that $M\simeq V_\lambda$ for some $\lambda\in\Theta$.

Let $\gg=o(\infty)$ or $sp(\infty)$. In the first case we assume that
$\gg_i=o(2i+1)$. Then
$\lambda_i=a^i_1\varepsilon_1+\dots+a^i_k\varepsilon_k$
for some fixed $k$ and $a^i_1\geq a^i_2\geq\dots\geq a^i_k\geq 0$. The
sequence $\{a^i_j\}$ is non-decreasing for every fixed $j$ as follows
from the branching laws for the respective pairs $o(2m+1)\subset o(2m+3)$
and $sp(2n)\subset sp(2m+2)$, see [GW]. Then by
repeating the argument in the previous paragraph we can prove that
$\{a^i_j\}$ stabilizes, and consequently $M\simeq V_\lambda$ for some $\lambda\in\Theta$.
\end{proof}

\refcor{tens} and \refth{slind} show that a simple module $M\in\Int_\gg$
is an object of $\lind_\gg$ if and only if $\Gamma_\gg(M^*)$ has finite
Loewy length. Below we will use this fact to give an equivalent 
definition of $\lind_\gg$ (Corollary 6.13). Furthermore, it
is easy to check (see also [PS]) that for sufficiently large
$i$ the simple $\gg_i$-module $V^i_\lambda$ occurs in $Y$ with
multiplicity 1, and all other simple $\gg_i$-constituents have
infinite multiplicity and are isomorphic to $V^i_{\mu}$ with
$|\mu|<|\lambda|$. In what follows we call this unique
$\gg_i$-constituent the {\it canonical $\gg_i$-constituent of}
$V_\lambda$. Note also that by Corollary \ref{simplesoc} for each
simple object $M$ of $\lind_\gg$, $M_*$ is a well-defined simple
object in $\lind_\gg$. Hence $M_*$ is well defined also for any
semisimple object $M$ of $\lind_\gg$: if $\ds
M=\bigoplus_{\lambda\in\Theta}M^\lambda\otimes V_\lambda$
($M^\lambda$ being trivial $\gg$-modules), then $\ds
M_*=\bigoplus_{\lambda\in\Theta}M^\lambda\otimes (V_\lambda)_*$.
It is clear that $M_*\cong M$ for $\gg\cong o(\infty),\,
sp(\infty)$.

\begin{corollary} The simple objects of $\lind_\gg$ are precisely the
  simple tensor modules.
\end{corollary}

\begin{lemma}\label{injhull} Let $M\cong V_\lambda$ be a simple tensor
 module. Then $\soc((M_*)^*)\simeq M$.
If $V_\mu$ is a subquotient of $(M_*)^*$ and $\mu\neq\lambda$,
then $|\mu|<|\lambda|$.
\end{lemma}
\begin{proof} The first statement follows from Corollary \ref{simplesoc}.

The second statement follows immediately from the fact that
$\hom_{\gg_i}(V_\mu^i,(M_*)^*)\neq 0$ implies $|\mu|<|\lambda|$.
\end{proof}

\begin{corollary}\label{injtenscat}
a) For any simple $M\in\lind_\gg$, $(M_*)^*$ is an injective
hull of $M$ in $\Int_\gg$ (and hence also in $\lind_\gg$).

b) Any indecomposable injective object in $\lind_\gg$ is
isomorphic to $M^*$ for some simple module $M\in\lind_\gg$. In
particular, any indecomposable injective module is isomorphic to a
direct summand of $(T^{p,q})^*$ for some $p,q$.

c) For any $M\in\lind_\gg$, any injective hull $I_M$ of $M$ in $\Int_\gg$
is an object of $\lind_\gg$.
\end{corollary}
\begin{proof} a) Follows directly from \refprop{inj} and Lemma \ref{injhull}.

b) To derive b) from a) it suffices to note that an injective module in
$\lind_\gg$ is indecomposable if and only if it has simple socle.

c) follows from the fact that $I_M$ is isomorphic to a submodule of
$\Gamma_\gg(M^{**})$, see Corollary 3.3
\end{proof}

In what follows we set $I_\lambda:=((V_\lambda)_*)^*$.
\begin{corollary}\label{brick} $\End_\gg(I_\lambda)=\CC$.
\end{corollary}
\begin{proof} If $\phi\in \End_\gg(I_\lambda)$, then
$\phi|_{V_\lambda}=c\,\Id$ for $c\in\CC$. Therefore
$V_\lambda\subset\Ker(\phi-c\,\Id)$. Furthermore, any non-zero
$\gg$-submodule of $I_\lambda$ contains $\soc(I_\lambda)=V_\lambda$, hence
$V_\lambda \subset\Im(\phi-c\,\Id)$. This implies $\phi-c\,\Id=0$,
as otherwise $V_{\lambda}$ would be isomorphic to a subquotient of $I_\lambda /
V_\lambda$ contrary to Lemma \ref{injhull}.
\end{proof}
\begin{lemma}\label{glue} Let $X,Y,Z,M\in\lind_\gg$. Assume furthermore that $Y$
is simple, $Y=\soc(M)$, and there exists an exact sequence $$0\to
X\to Z \stackrel{p}{\to} Y\to 0.$$ Then there exists $\tilde
{M}\in\Int_\gg$ such that $Z\subset \tilde{M}$  and
  $\tilde{M}/X\simeq M$.
\end{lemma}
\begin{proof} Let $Y_i$ be the canonical $\gg_i$-constituent of
  $Y$. Then $Y=\underrightarrow{\lim}Y_i$. Set $Z_i:=p^{-1}(Y_i)$ and
  $Q_i:=Z_i\cap X$. Then
$Z_i=Y_i\oplus Q_i$ and there are injective homomorphisms
$\phi_i:Z_i\to Z_{i+1}$
$$\phi_i(y,q)=(e_i(y),t_i(y)+f_i(q)),\,\,y\in Y_i, q\in Q_i$$
for some non-zero homomorphisms $e_i:Y_i\to Y_{i+1}$, $t_i:Y_i\to
Q_{i+1}$ and $f_i:Q_i\to Q_{i+1}$. Clearly,
$Z=\underrightarrow{\lim}Z_i$.

On the other hand, $M=\underrightarrow{\lim}M_i$ for some nested
finite-dimensional $\gg_i$-submodules $M_i\subset M$ such that
$Y_i\subset M_i$. Moreover, $\dim\hom_{{\gg}_i} (Y_i,M_i)=1$ by Lemma
\ref{injhull}. Therefore, $M_i$ has a unique $\gg_i$-module decomposition
$M_i=R_i\oplus Y_i$. The inclusions $\psi_i:M_i\to M_{i+1}$  are
given by
$$\psi_i(r,y)=(p_i(r),s_i(r)+e_i(y)),\,\,y\in Y_i, r\in R_i$$
for some non-zero homomorphisms $p_i:R_i\to R_{i+1}$ and
$s_i:R_i\to Y_{i+1}$.

Define $\tilde{M}_i:=R_i\oplus Y_i\oplus Q_i$ and let
$\zeta_i:\tilde{M}_i\to \tilde{M}_{i+1}$ be given by the formula
$$\zeta(r,y,q)=(p_i(r),s_i(r)+e_i(y),t_i(y)+f_i(q)).$$
Set $\tilde{M}:=\underrightarrow{\lim}\tilde{M}_i$. It is easy to
check that $\tilde{M}$ satisfies the conditions of the lemma.
\end{proof}

\begin{lemma} \label{surj} If $\hom_\gg(I_\lambda, I_\mu)\neq 0$, then
  $|\mu|\leq|\lambda|$. If $I$ is any injective object of $\lind_\gg$
and $0\neq\phi\in \hom_\gg(I, I_\mu)$, then $\phi$ is surjective.
\end{lemma}
\begin{proof} The first statement follows immediately from Lemma
  \ref{injhull}.

To prove the second statement put $X=\Ker \phi$, $Y=V_{\mu}$,
$Z=\phi^{-1}(Y)$ and $M=I_\mu$. Construct $\tilde{M}$ as in Lemma
\ref{glue}. By the injectivity of $I$, the injective
homomorphism $Z\to \tilde{M}$ extends to a homomorphism
$\tilde{M}\to I$. The latter induces a homomorphism $\eta:
M=I_\mu\to I/X$.

Let now $\bar{\phi}:I/X\to I_{\mu}$ denote the injective
homomorphism induced by $\phi$. Then it is obvious that
$\bar{\phi}\circ\eta(y)=y$ for any $y\in Y$. By Corollary \ref{brick},
we have $\bar{\phi}\circ\eta=\Id$. Hence $\bar{\phi}$ is an
isomorphism, i.e. $\phi$ is surjective.
\end{proof}

\begin{prop}\label{linj} The Loewy length of $I_\lambda$ equals $|\lambda|+1$.
\end{prop}
\begin{proof} By Lemma \ref{injhull} we know that the  Loewy length of
  $I_\lambda$ is at most $|\lambda|+1$. We prove equality by
  induction in $|\lambda|$. Fix $\mu\in\Theta$ such that $|\mu|=|\lambda|-1$ and
$\hom_{\gg_i}(V^i_\mu,V^{i+1}_\lambda)\neq 0$. We claim that
$\ext^1  (V_\mu,V_\lambda)\neq 0$. Indeed, consider non-zero
homomorphisms $\phi_i\in \hom_{\gg_i}(V^i_\mu,V^{i+1}_\lambda)$.
Set $X=\underrightarrow{\lim}X_i$, where  $X_i=V^i_\mu\oplus
V^i_\lambda$, $q_i:X_i\to X_{i+1}$ is given by
$q_i(x,y)=(e_i(x),\phi_i(x)+f_i(y))$ for $x\in V_\mu$, $y\in
V_\lambda$, and $e_i:V^i_\mu\to V^{i+1}_\mu$
  and $f_i:V^i_\lambda\to V^{i+1}_\lambda$ denote the fixed inclusions.
It is easy to see that $X$ is a non-trivial extension of $V_\mu$ by
  $V_\lambda$.

Thus, we have a non-zero homomorphism $I_\lambda\to I_\mu$. By
Lemma \ref{surj}, it is surjective. Hence the Loewy length of
$I_\lambda$ is greater or equal to the Loewy length of $I_\mu$ plus 1. The
statement follows.
\end{proof}

The following theorem strengthens the claim of Corollary \ref{tens}.

\begin{theo}\label{subquotients} Let $M\in\Int_{\gg}$. Then $M\in\lind_\gg$
if and only if there
  exists a finite subset $\Theta_M\subset\Theta$ such that any simple subquotient
of $M$ is isomorphic to $V_\mu$ for $\mu\in \Theta_M$.
\end{theo}

\begin{proof}
Assume that $M\in\lind_\gg$. It is sufficient to prove the 
existence of $\Theta_M$ for a semisimple
  $M$ since then the general case follows from Lemma \ref{injhull}.
Without loss of generality we may assume that $M= \bigoplus_{j\in
C} V_{\lambda_j}$, where $V_{\lambda_j}$ are pairwise
non-isomorphic. We claim that if $C$ is infinite, then $M^*$ does
not have finite Loewy length. Indeed, $M^*$ contains a submodule
isomorphic to $\bigoplus_{j\in C} I_{\mu_j}$, where
$V_{\mu_j}=(V_{\lambda_j})_*$. If $C$ is infinite, then
$|\mu_j|=|\lambda_j|$ is unbounded and the socle filtration of
$\bigoplus_{j\in C} I_{\mu_j}$ is infinite. This proves one direction.

Now assume that $M$ admits a finite set $\Theta_M$ as in the statement
of the theorem. We claim first that if $M'$ is a quotient of $M$ and
$\ext_\gg^1(M',V_\lambda)\neq 0$ for some $\lambda\in\Theta$, then $M$
has a subquotient isomorphic to $V_\mu$ for some $\mu <\lambda$. Indeed,
by extending the sequence $0\rightarrow V_\lambda\rightarrow
I_\lambda$ to a minimal injective resolution $0\rightarrow V_\lambda
\rightarrow I_\lambda\stackrel{i}{\rightarrow}I_\lambda^1\rightarrow ...$,
we see that there is a non-zero homomorphism $M'\stackrel{p}{\rightarrow} I_\lambda^1$.
Furthermore, by the minimality of the resolution, we have $\soc(I_\lambda^1)\subset\im i$.
Hence by Lemma 6.9 every simple constituent of $\soc(I_\lambda^1)$ is
of the form $V_\nu$ for $\nu <\lambda$. Since $(\im p)\cap\soc(I_\lambda^1)\neq 0$,
some simple constituent of $\soc(I'_\lambda)$ is isomorphic to a subquotient of $M'$
and thus of $M$.

We show now that $M$ has finite Loewy length. Consider a minimal (with respect
to the order $\leq$) weight $\lambda\in\Theta$. The above argument shows that
$\ext_\gg^1(M',V_\lambda)=0$ for any quotient $M'$ of $M$. This implies that
every subquotient of $M$ isomorphic to $V_\lambda$ is a quotient of $M$.
Hence $M$ admits a surjective homomorphism $\zeta :M\rightarrow M_\lambda$,
where $M_\lambda$ is isomorphic to a direct sum of copies of $V_\lambda$
and $\Theta_{\ker\zeta}=\Theta_M\setminus\{\lambda\}$. By an induction argument
we obtain that $M$ has finite Loewy length. Therefore $M\in\lind_\gg$
by Corollary \ref{tens} c).
\end{proof}
\begin{corollary}\label{agmod}
A $\gg$-module $M\in\Int_\gg$ is an object of $\lind_\gg$ if and only if both
$M$ and $\Gamma_\gg(M^*)$ have finite Loewy length.
\end{corollary}
\begin{proof}
In one direction the statement is trivial. We need to prove that, 
if $M\in\Int_\gg$ satisfies the above two conditions, then $M^*\in\Int_\gg$.
For a semisimple $M$ this follows directly from Theorem \ref{subquotients} (as we have already
pointed out). The argument gets completed by induction on the Loewy length. Let $M \in\Int_\gg$
have Loewy length $k$, and $\Gamma_\gg(M^*)$ have finite Loewy length.
Consider the homomorphism $\pi:M \rightarrow \ \Top(M)$ onto the maximal semisimple
quotient $\Top(M)$ of $M$. Then $\Gamma_\gg((\Top(M))^*)\subset\Gamma_\gg(M^*)$, hence
$\Top(M)\in\lind_\gg$, i.e. in particular $(\Top(M))^*\in\Int_\gg$. Therefore there is an exact 
sequence $$0 \rightarrow (\Top(M))^* \rightarrow\Gamma_\gg(M^*)\rightarrow\Gamma_\gg((\Ker\pi)^*)
\rightarrow 0,$$
implying that $\Gamma_\gg((\Ker\pi)^*)$ has finite Loewy length. Since the Loewy length
of $\Ker\pi$ equals $k-1$, we can conclude that $(\Ker\pi)^* \in\Int_\gg$. Hence
$\Gamma_\gg(M^*)=M^*$.
\end{proof}
\begin{corollary}
$\lind_\gg$ is a tensor category with respect to $\otimes$.
\end{corollary}
\begin{proof}
It suffices to show that $\lind_\gg$ is closed with respect to $\otimes$.
The fact that, if $M\in\lind_\gg$ and $M'\in\lind_\gg$ then $M\otimes M'\in\lind_\gg$,
follows immediately from Theorem \ref{subquotients}.
\end{proof}

The following theorem concerns the structure of injective modules in $\lind_\gg$.
\begin{theo}\label{injfil}
Any injective module $I\in\lind_\gg$ has a finite filtration $\{I_j\}$ such
that, for each $j$, $I_{j+1}/I_j$ is isomorphic to a direct sum of copies of
${I_{\mu}}_j$ for some $\mu_j\in\Theta$.
\end{theo}
\begin{proof}
We use induction on the length of the filtration. Assume that
$0=I_0\subset I_1\subset I_k$ 
is already constructed. Let $\soc(I/I_k)=\bigoplus_{f\in\mathcal
F}Y_f$ for a family $\mathcal F$ of simple modules $Y_f$ (there are
only finitely many non-isomorphic modules among $\{Y_f\}_{f\in\mathcal F}$).
Denoting by $p$ the projection $\mu_f:I\rightarrow I/I_k$, set
$X_f:=p^{-1}(Y_f)$. By Lemma \ref{glue}, there exists $\tilde{Y_f}\in\Int_\gg$
such that $I_k\subset X_f\subset\tilde{Y_f}$ and $\tilde{Y_f}/I_k\simeq I_{\mu_f}$,
$\mu_f\in\Theta$ being the highest weight of $Y_f$. The inclusion $X_f\subset I$
induces a homomorphism $\psi_f:\tilde{Y_f}\rightarrow I$. Let
$\psi_f:\tilde{Y_f}/I_k\tilde{\rightarrow}{I_\mu}_f\rightarrow I/I_k$
the corresponding homomorphism of quotients. Then
$\bar\psi :=\bigoplus_{f\in\mathcal F}\bar\psi_f:\bigoplus_{f\in\mathcal F}I_{\mu_f}\rightarrow I$
is injective since its restriction to $\soc(\bigoplus_{f\in\mathcal F}I_{\mu_f})$
is an isomorphism. This shows that if $I_{k+1}:=p^{-1}(\bar\psi(\bigoplus_{f\in\mathcal F}I_{\mu_f}))$,
there is an isomorphism $I_{k+1}/I_k\simeq\bigoplus_{f\in\mathcal F}I_{\mu_f}$.

The filtration terminates at a finite step as $I$ has finite Loewy length.
\end{proof}
\begin{example}
Let $\gg=sl(\infty),o(\infty),sp(\infty)$ and let $M$ be a countable
direct sum of copies of $V$, i.e. $M=\bigoplus_{f\in\mathcal F}V_f$,
$\mathrm{card} \mathcal F = \beth_0$. Then $(M_*)^*$ can be identified
with the set of all infinite matrices ${\{b_{ij}\}_{i,j\in\mathbb{Z}}}_{>0}$,
the action of $\gg$ being left multiplication. The socle $\soc((M_*)^*)$
is the space of matrices $F_r$ with finitely many non-zero rows and is
isomorphic to $\bigoplus_{g\in 2^{\mathcal F}}V_g$. (Note that the
module $\prod_{f\in\mathcal F}V_f$ considered in Example 6.2 d) is a
submodule of $(M_*)^*$ and has the same socle as $(M_*)^*$). We thus obtain
the diagram 

$$\begin{array}[c]{ccc}
\bigoplus_{g\in 2^{\mathcal F}}V_g&\subset &(M_*)^*\\
\cup &&\cup \\
M&\subset &I_M
\end{array},$$

\noindent $I_M$ being the injective hull of $M$ within $(M_*)^*$. Moreover, $I_M$ is the
largest submodule of $(M_*)^*$ such that $\gg\cdot I_M=M$. A direct computation
shows that $I_M$ coincides with the space of all matrices with finite rows
(i.e. each row has finitely many non-zero entries).

Note that $I_M\not\simeq\bigoplus_{f\in\mathcal F}(I_{\varepsilon_1})_f$
($\varepsilon_1\in\Theta$ is the highest weight of $V$). In fact
$I_M$ has the following filtration as in Theorem \ref{injfil}:
$0\subset\bigoplus_{f\in\mathcal F}(I_{\varepsilon_1})_f\subset I_M$.
Here $I_M/\bigoplus_{f\in\mathcal F}(I_{\varepsilon_1})_f$ is a trivial
module of cardinality $2^{\mathcal F}$ which is interpreted as a
direct sum of $2^{\mathcal F}$ copies of $I_0$.
\end{example}

For any $k\in\ZZ_{>0}$ we now define $\lind^k_\gg$ be the subcategory
of modules whose simple quotients are isomorphic to $V_\mu$ with
$|\mu|\leq k$. Theorem \ref{subquotients} and Corollary \ref{tens} a) imply
the following.
\begin{corollary}\label{dircat} The category $\lind^k_\gg$ is closed
under direct products and direct sums.
\end{corollary}
\begin{corollary}
a) The category $\lind_\gg$ equals the direct limit $\underrightarrow{\lim}
\lind_\gg^k$.

b) If $\{M_f\}_{f\in\mathcal F}$ is an infinite family of objects of
$\lind_\gg$, then $\prod_{f\in\mathcal F}M_f\in\lind_\gg$ (equivalently,
$\bigoplus_{f\in\mathcal F}\in\lind_\gg$) if and only if there is $k$ such that $M_f\in\lind_\gg^k$
for all $f\in\mathcal F$.
\end{corollary}
\begin{proof}
a) follows directly from Theorem \ref{subquotients}.

Consider now $\prod_{f\in\mathcal F}M_f$.
If $M_f\in\lind_\gg^k$ for some $k$, then 
$\prod_{f\in\mathcal F}M_f\in
\lind_\gg^k$ (and thus also $\bigoplus_{f\in\mathcal  F}M_f\in\lind_\gg^k$) 
by Corollary \ref{tens} a). If no such $k$
exists, then $\bigoplus_{f\in\mathcal F} M_f\notin\lind_\gg$ 
by Theorem \ref{subquotients}, hence also $\prod_{f\in\mathcal F}M_f\notin\lind_\gg$.
\end{proof}
\begin{corollary}\label{resolution} Every object in $\lind_\gg$ has a
 finite injective resolution.
\end{corollary}

We now introduce the following partial order on $\Theta$: we set
$\mu\leq\lambda$ if for any sufficiently large $i$ there exists
$j>i$ such that $\hom_{\gg_i}(V^i_\mu,V^j_\lambda)\neq 0$. If
$\mu\leq\lambda$, then $l(\lambda,\mu)$ denotes the length of a
maximal chain $\mu<\mu_1<\dots<\lambda$ in $\Theta$.

\begin{lemma}\label{extsimp} $\ext^1_\gg(V_\mu,V_\lambda)\neq 0$ if and only if $\mu<\lambda$. If $\mu<\lambda$, $\dim\ext^1_\gg(V_\mu,V_\lambda)=2^\ZZ$.
\end{lemma}
\begin{proof} Assume that there is a non-trivial extension
\begin{equation}\label{eq6}
0\to V_\lambda\to X\to V_\mu\to 0.\end{equation} We will show that
$\mu<\lambda$. Let, on the contrary,
$\hom_{\gg_i}(V^i_\mu,V^j_\lambda)=0$ for all $j>i$. Then
$\hom_{\gg_i}(V^i_\mu,V_\lambda)=0$. Since
$\dim\hom_{\gg_i}(V^i_\mu,V_\mu)=1$, we have
$\dim\hom_{\gg_i}(V^i_\mu,X)=1$. Let $\phi:V^i_\mu\to X$ be a
non-zero homomorphism. Then $U(\gg)\cdot\phi(V^i_\mu)\simeq X$.
Therefore $\phi$ extends to a homomorphism of $\gg$-modules
$V_\mu\to X$, and this yields a splitting of the sequence in
\refeq{eq6}. Thus, $\ext^1_\gg(V_\mu,V_\lambda)\neq 0$ implies
$\mu<\lambda$.

Now let $\mu<\lambda$. Then there exists an infinite sequence
$i_1,i_2,\dots$ such that $\hom_{\gg_{i_j}}(V^{i_j}_\mu,
V^{i_{j+1}}_\lambda)\neq 0$ for all $j$. Consider a sequence of
non-zero homomorphisms $\phi_j\in
\hom_{\gg_{i_j}}(V^{i_j}_\mu,V^{i_{j+1}}_\lambda)$ and set
$Z_j:=V^{i_j}_\mu\oplus V^{i_j}_\lambda$. Denote by $e_j$
(respectively, $f_j$) the inclusion $V^{i_j}_\mu\to
V^{i_{j+1}}_\mu$ (resp., $V^{i_j}_\lambda\to
V^{i_{j+1}}_\lambda$). Define $\psi_j:Z_j\to Z_{j+1}$ by
$$\psi(x,y)=(e_j(x),\phi_j(x)+f_j(y)),\,\,x\in V^{i_j}_\mu, y\in V^{i_{j}}_\lambda.$$
Consider $Z=\underrightarrow{\lim} Z_j$. It is an exercise to
check that $Z$ is an extension of $V_\mu$ by $V_\lambda$, and it
does not split if infinitely many $\phi_j\neq 0$. Hence the
dimension of $\ext^1_\gg(V_\mu,V_\lambda)$ is at least $2^\ZZ$. On
the other hand, the dimension of $\ext^1_\gg(V_\mu,V_\lambda)$ is
bounded by the multiplicity of $V_\mu$ in
$\soc^1(I_\lambda)/\soc(I_\lambda)$. The dimension of
$I_\mu=((V_\mu)_*)^*$ is $2^\ZZ$, hence the dimension of
$\ext^1_\gg(V_\mu,V_\lambda)$ is at most $2^\ZZ$.

To finish the proof just note that
$\ext^1_\gg(V_\lambda,V_\lambda)=0$ by Lemma \ref{injhull}.
\end{proof}
\begin{corollary}
The category $\lind_\gg$ consists of a single block.
\end{corollary}
\begin{proof}
According to Lemma \ref{extsimp}, $\ext_\gg^1(\mathbb{C},V_\mu)\neq 0$ for
any $\mu\in\Theta$.
\end{proof}
\begin{prop}\label{socfilt} For $k\in\ZZ_{>0}$, set
$$\Theta^k(\lambda)=\{\mu<\lambda | l(\lambda,\mu)\geq k+1\}.$$
Then
$$\soc^k(I_\lambda)/\soc^{k-1}(I_\lambda)=\bigoplus_{\mu\in\Theta^k(\lambda)}X^\mu\otimes
V_\mu,$$ where each $X^\mu$ is a trivial $\gg$-module of dimension
$2^\ZZ$.
\end{prop}
\begin{proof} For $k=1$ the statement follows from Lemma  \ref{extsimp}.
Now we proceed by induction on $k$. Note first that if $V_\mu$ is
a simple constituent of $\soc^k(I_\lambda)/\soc^{k-1}(I_\lambda)$,
then, by \refle{extsimp}, $\mu<\chi$ for some simple constituent
$V_\chi$ of $\soc^{k-1}(I_\lambda)/\soc^{k-2}(I_\lambda)$. By the
induction assumption, $\chi\in\Theta^{k-1}(\lambda)$. In addition,
it is clear that $V_\mu$ is a simple constituent of
$\soc^k(I_\lambda)/\soc^{k-1}(I_\lambda)$ if and only if there
exists a non-zero homomorphism $\phi:I_\lambda\to I_\mu$, such
that $\phi(\soc^{k-1}(I_\lambda))=0$. By Lemma \ref{surj}, $\phi$
is surjective, so all simple constituents of
$\soc^1(I_\mu)/\soc(I_\mu)$ are also simple constituents of
$\soc^k(I_\lambda)/\soc^{k-1}(I_\lambda)$. This implies that
$V_\mu$ is a simple constituent of
$\soc^k(I_\lambda)/\soc^{k-1}(I_\lambda)$ if and only if there
exists $\psi\in\Theta^{k-1}(\lambda)$ such that
$\mu\in\Theta^1(\psi)$. Since $\mu\in\Theta^1(\psi)$ if and only
if $\mu\in\Theta^k(\lambda)$, the statement follows.
\end{proof}

Let $\tens_\gg$ be the full subcategory of $\lind_\gg$ consisting of
modules $M$ whose cardinality $\mathrm{card}M$ is
bounded by $\beth_n$ for some $n$ depending on $M$.

\begin{theo}\label{cardinality} $\tens_\gg$ is the unique minimal abelian full subcategory of $\Int_\gg$ which does not consist of trivial modules only and which is closed under $\otimes$ and $^*$.
\end{theo}
\begin{proof} Let $\mathcal C$ be a minimal abelian full subcategory of $\Int_\gg$ which contains a non-trivial module $M$ and is closed under $\otimes$ and $^*$.
We will show that $V\in\mathcal C$. Since $\End_{\CC}M$ is a
$\gg$-submodule of $(M^*\otimes  M)^*$ (through the map
$\phi(\psi\otimes m)=\psi(\phi(m))$ for $m\in M,\,\psi\in
M^*,\,\phi\in\End_\CC M$), we have $\End_{\CC}M\in\mathcal C$.
Furthermore, the adjoint module $\gg$ is a
  submodule of $\End_{\CC}M$. Hence $\gg\in\mathcal C$. Recall that
  $\gg$ is the socle of $V_*\otimes V$ for $sl(\infty)$, of $\Lambda^2(V)$
  for $o(\infty)$, and of $S^2(V)$ for $sp(\infty)$. In all cases it is
  easy to see that $\gg^*$ contains a subquotient isomorphic to $V$.
  Therefore $V\in\mathcal C$. In addition, $V_*=\soc(V^*)\in\mathcal C$.
  Therefore $T^{p,q}\in\mathcal C$ for all $p,q$, and $V_\lambda\in\mathcal
  C$ for all $\lambda\in\Theta$. Finally, by \refcor{injtenscat} a), any $M\in\lind_\gg$ is a
  submodule of $(\soc(M)_*)^*$, and the
  statement follows.
\end{proof}

We conclude this paper with the remark that the category
$\lind_\gg$, for $\gg=sl(\infty), o(\infty), sp(\infty)$, is
functorial with respect to any  homomorphism of locally semisimple
Lie algebras $\phi:\gg'\to\gg$. By this we mean that any
$M\in\lind_{\gg}$ considered as a 
$\gg'$-module is an object of
$\lind_{\gg}$.

To prove this, recall that the image of $\phi'$, being a locally
semisimple subalgebra of $\gg$, is isomorphic to a direct sum of
copies of $sl(\infty), o(\infty), sp(\infty)$ and of
finite-dimensional simple Lie algebras, [DP2]. Furthermore,
the result of [DP2] implies that as
$\gg'$-modules both $V$ and $V_*$ have Loewy length at most 2 and
that all non-trivial simple constituents of $V$ and $V_*$ are
isomorphic to the natural and conatural representations
$V_{\mathfrak  s}$ and $(V_{\mathfrak  s})_*$ for some simple
direct summands $\mathfrak s$ of $\phi(\gg')$ and that all non-trivial
constituents occur with
finite multiplicity. (The simple trivial representation may occur
with up to countable multiplicity in both $\soc(V)$ and
$V/\soc(V)$ (respectively, $\soc(V_*)$ and $V_*/\soc(V_*)$.)
This allows us to conclude that any single simple object of
$\lind_\gg$ is an object of $\lind_{\phi(\gg')}$. Hence, by Theorem \ref{subquotients},
any $M\in\lind_\gg$ is an object of $\lind_{\phi(\gg')}$.

\bigskip

\textbf{References}

\bigskip
[Ba1] A. Baranov, Complex finitary simple Lie algebras, Archiv der Math. 71(1998), 1-6.

\bigskip

[Ba2] A. Baranov, Simple diagonal locally finite Lie algebras,  Proc. London Math. Soc. (3)  77  (1998),  no. 2, 362-386

\bigskip

[Ba3] A. Baranov, Finitary simple Lie algebras, J. Algebra 219(1999), 299-329.

[BZh] A. Baranov, A. Zhilinski, Diagonal direct limits of simple Lie algebras, Comm.
Algebra 27 (1998), 2749-2766.

\bigskip

[DP] I. Penkov, I. Dimitrov, Weight modules of direct limit Lie algebras, IMRN 1999, no. 5, 223-249.

\bigskip

[GW] R. Goodman, N. Wallach, Representations and Invariants of the Classical Groups,
Cambridge University Press, 1998.

\bigskip

[PS] I. Penkov, K. Styrkas, Tensor representations of infinite-dimensional root-reductive Lie algebras, in Trends and Developments in Infinite Dimensional Lie Theory, Birkhauser, to appear.

\bigskip

[PStr] I. Penkov, H. Strade, Locally finite Lie algebras with root
decomposition, Archiv der Math. 80 (2003), 478-485.

\bigskip

[W] C. Weibel, An introduction to homological algebra, Cambridge Studies in Advanced Mathematics, Cambridge University Press, 1994.
\end{document}